\documentclass[a4paper,12pt]{article}

\usepackage{amsmath,amsthm,amssymb,sansmath,color,geometry,multicol,graphicx, float,color,authblk, hyperref}
\usepackage[english]{babel}
\usepackage[utf8]{inputenc}
\newtheorem{thm}{Theorem}
\newtheorem{prop}[thm]{Proposition}

\newtheorem{lem}[thm]{Lemma}
\newtheorem{cor}[thm]{Corollary}

\newtheorem{defi}[thm]{Definition}

\newcommand{\p}{\partial}
\newcommand{\x}{\partial_x}

\newcommand{\po}{\left(}
\newcommand{\pf}{\right)}
\newcommand{\co}{\left[}
\newcommand{\cf}{\right]}
\newcommand{\R}{\mathbb R}
\newcommand{\E}{\mathbb E}

\newcommand{\A}{\mathcal A}

\newcommand{\na}{\nabla}
\newcommand{\ent}{\text{Ent}_p}

\title{On $\mathcal H^1$ and entropic convergence for contractive PDMP.}
\author{Pierre Monmarché}
\affil{Institut de Mathématiques de Toulouse}

\begin{document}

\maketitle
\abstract{Explicit rate of convergence in variance (or more general entropies) is obtained for a class of Piecewise Deterministic Markov Processes such as the TCP process, relying on functional inequalities. A method to establish Poincar\'e (and more generaly Beckner) inequalities with respect to a diffusion-type energy for the invariant law of such hybrid processes is developped.}

\section{Introduction}

This work is devoted to the study of convergence to equilibrium for a class of Piecewise Deterministic Markov Process (PDMP). These hybrid processes, satisfying a deterministic differential equation between random jumps, have received much attention recently: we refer to  \cite{Krell} and the references therein for an overview of the topic. Ergodicity and, then, speed of convergence to the steady state are particularly studied. As far as this last point is concerned, coupling methods have recently proved efficient in order to get explicit rate of convergence in Wasserstein distances for PDMP (see \cite{ChafaiMalrieuParoux,Malrieu2011,BLBMZ2,Fontbona2010,Bouguet} for instance, among many others).  On the other hand, another classical approach to quantify ergodicity, based on functional inequalities, is hardly used, since the usual methods do not directly apply. Our aim is to adapt them (see also \cite{LaurencotPerthame} in this direction).

Let $\Omega$ be an open set of $\R^d$. The dynamics is defined thanks to a vector field $b:\Omega\rightarrow \R^d$, a jump rate $\lambda : \Omega \rightarrow \R_+$, and a transition kernel $Q$ which will be seen either as a function from $\Omega$ to $\mathcal P\po \Omega\pf$ the set of probability measures on $\Omega$, or as an operator on some functional space. 
For $x\in \Omega$ let $\po \varphi_x(t)\pf_{t\geq0}$ be the flow associated to $b$, namely the solution of
\[\p_t \varphi_x(t) = b\po \varphi_x(t)\pf,\hspace{20pt} \varphi_x(0) = x.\]
Starting at point $x$, the process $(X_t)_{t\geq 0}$ deterministically follows this flow up to its first jump time $T_x$ with law
\[\mathbb P\po T_x< s\pf =  \int_0^s \lambda\po \varphi_x(u)\pf e^{-\int_0^u \lambda\po \varphi_x(w)\pf dw} du = 1 - e^{-\int_0^s \lambda\po \varphi_x(w)\pf dw}.\]
At time $T_x$, the process jumps according to the law $Q\po \varphi_x(T_x)\pf$, and starts anew from its new position. The infinitesimal generator of the process is
\begin{eqnarray}\label{DefinitionGenerateur}
Lf(x)  & = & b(x).\na f(x)  + \lambda(x)\po Q f(x) - f(x)\pf,
\end{eqnarray}
defined at least for bounded $f\in\mathcal{C}^1\po \Omega\pf$. We note 
\[P_t f (x) = \mathbb E\po f(X_t) | X_0 = x\pf\]
the associated semi-group. The following assumptions hold throughout this work: \begin{itemize}
\item the flow is well-defined and it fixes $\Omega$: if $x\in \Omega$ then $\varphi_x(t) \in \Omega$ for all $t>0$.
\item the process is non-explosive: there can't be infinitely many jumps in a finite time interval, so that the process (and therefore the semi-group) is defined for all time.  We suppose $\lambda >0$ almost everywhere, and on every fixed point of the flow.
\item the functions $\lambda$ and $b$ are smooth; we write $J_b(x) = \co \p_i b_k (x)\cf_{1\leq i,k \leq d}$ the Jacobian matrix of $b = \po b_k\pf_{1\leq k \leq d}$.
\item The process admits a unique invariant law $\mu$, and $P_t$ is ergodic in the sense $P_t f(x) \underset{t\rightarrow \infty}\longrightarrow \int f \text{d} \mu$ 
for all $f\in L^2(\mu)$ and all $x\in \Omega$. Moreover all polynomial moments of $\mu$ are finite and, denoting by 
$\A$ the set of function in $\mathcal C^\infty(\Omega)$ whose derivatives grow at most polynomially at infinity, $Q$, $L$ and $(P_t)_{t\geq0}$ are well-defined on $\A$ and they fixes $\A$.
\end{itemize}
These strong assumptions allow us to focus only on the quantification of ergodicity.  Note that the uniqueness of the invariant measure, the finiteness of its moments and the ergodicity of the process may often be proved by checking it is irreducible and admits a Lyapunov function (cf. \cite{MeynTweedieDown}). Throughout this work the test functions will always belong to the set $\A$, in order to keep the study at a formal level, all the forthcoming elementary definitions and calculations being licit in this framework.

\bigskip

We recall here some classical arguments (see \cite[Chapter 5]{Logsob} for a general introduction to functional inequalities and for the detailed proofs of the assertions in this paragraph). For $f\in\A$, we write $\Gamma \po f\pf = \frac12 L(f^2) -  f L f$ the \emph{carré du champ} operator of $L$, $\Gamma(f,g)$ the corresponding symetric bilinear operator obtained by polarization, and 
\[\Gamma_2(f) = \frac12 L\po \Gamma f\pf - \Gamma\po f,Lf\pf.\]
 Writing $\psi(s) = P_s\Gamma \po P_{t-s} f\pf$, from $\partial_t P_t f = LP_t f = P_t Lf$ one gets
\[\psi'(s) = 2 P_s\Gamma_2 \po P_{t-s} f\pf.\]
Hence, if the Bakry-Emery (or $\Gamma_2$) criterion $\Gamma_2 > \rho \Gamma$ holds for some $\rho>0$, the Gronwall Lemma yields $\psi(0) \leq e^{-2\rho t} \psi(t)$, namely
\begin{eqnarray}\label{EquationCommuteGamma}
\Gamma \po P_t f \pf & \leq & e^{-2\rho t}P_t \Gamma f.
\end{eqnarray}
For instance for the Ornstein-Uhlenbeck process with generator
\[L_{OU} f(x) = \Delta f(x) - \rho x \cdot \nabla f(x),\]
this reads
\begin{eqnarray}\label{EquationCommuteGradient}
|\nabla P_t f |^2 & \leq & e^{-2\rho t}P_t |\nabla f|^2,
\end{eqnarray}
where $|.|$ is the euclidian norm of $\R^d$. In fact, the sub-commutation \eqref{EquationCommuteGamma} is equivalent to the Bakry-Emery criterion. Nevertheless the latter does not usually hold in our settings. That said, a simple adaptation of the $\Gamma_2$ argument will give, at least in the constant jump rate case, a gradient estimate similar to \eqref{EquationCommuteGradient}. In the following we denote by $A^*$ the usual transpose of a matrix $A$ and thus by $u^*v$ the scalar product of two vectors.

\begin{thm}\label{ThmConvergencePonctuellePasPoids}
 Assume $\lambda$ is constant and $|\na Q f(x)|^2 \leq M(x) Q|\na f|^2(x)$ with $M$ such that
 \begin{eqnarray}\label{EqBalanceTauxConstant}
 \forall (x,u)\in \Omega\times \R^d,\hspace{27pt}2u^*J_b(x)u + \lambda\po M(x)-1\pf| u|^2  \leq - \eta | u|^2 
 \end{eqnarray}
 for some $\eta \in\R$. Then for all $t>0$, $f\in\A$ and $x\in\Omega$,
 \begin{eqnarray}\label{EqEstimGrad}
 |\na P_t f|^2(x) \leq e^{-\eta t} P_t |\na f|^2(x).
 \end{eqnarray}  
\end{thm}

Inequality \eqref{EqBalanceTauxConstant} is a balance condition on the drift and the jumps, reminiscent of the condition on the curvature in \cite[Theorem 1.2]{Cloez2012}. More precisely, suppose $|\na Q f(x)|^2 \leq M(x) Q|\na f|^2(x)$ for some function $M$ on $\Omega$. If $M<1$, $Q$ is a contraction of the Wasserstein distance (this will be detailed in Section \ref{SectionCVH1}); it means two particles that simultaneously jump can be coupled so that they get closer. More generaly $M$ measures how two such particles can be coupled in order for them not to get too far away one from the other. On the other hand, $J_b$ measures how two trajectories of the deterministic flow tends to get closer or to drift appart. Indeed,
\begin{eqnarray*}
\varphi_x(t) - \varphi_y(t) &  =&  x- y + t J_b(x)(x-y) + t\underset{y\rightarrow x}o(x-y) + \underset{t\rightarrow 0}o(t) \\
\Rightarrow\hspace{15pt}|\varphi_x(t) - \varphi_y(t)|^2 & = & |x-y|^2 + 2 t(x-y)^*J_b(x)(x-y)+ t\underset{y\rightarrow x}o\po|x-y|^2\pf + |x-y| \underset{t\rightarrow 0}o(t),
\end{eqnarray*}
 We see that the condition $u^* J_b(x) u < 0$ for all $(x,u)\in \Omega\times \R^d$ implies the flow contracts the space in the neighborhood of all points of $\Omega $.

\bigskip

Note that by integrating Inequality \eqref{EqEstimGrad} with respect to $\mu$ and writing
\[W_t = \int |\na P_t f|^2 d \mu, \]
Theorem \ref{ThmConvergencePonctuellePasPoids} implies $W_t \leq e^{-\eta t} W_0$ for all $t>0,\ f\in\A$, which is equivalent to $\p_t W_t \leq -\eta W_t$ for all $t>0,\ f\in\A$, or to $\po \partial_t W_t  \pf_{t=0} \leq -\eta W_0$ for all $f\in\A$. 

In the non-constant jump rate case, under a condition similar to \eqref{EqBalanceTauxConstant}, we will prove there exist constants $\beta>0$ and $\eta\in\R$ such that
\begin{eqnarray}\label{EqWtEvolution}
\p_t W_t \leq - \eta W_t + 2\beta \mathcal E_t
\end{eqnarray}
where $\mathcal E_t$ is defined as
\begin{eqnarray*}
\mathcal E_t & = & \int \Gamma \po P_t f\pf d\mu.
\end{eqnarray*}
Both $W_t$ and $\mathcal E_t$ are usually called energy ; we may say $W_t$ is the classical (or diffusion-like) energy, while $\mathcal E_t$ is the markovian one. They coincide in the case of the Ornstein-Uhlenbeck process.
The markovian energy usually appears in particular when one is concerned with the variance of $P_t f$ with respect to $\mu$,
\begin{eqnarray*}
V_t & = & \int (P_t f)^2 d\mu - \po\int P_t fd\mu\pf^2.
\end{eqnarray*}
We say $\mu$ satisfies a Poincar\'e (or spectral gap) inequality with respect to $\Gamma$ if there exist a constant $c>0$ such that $V_0 \leq c \mathcal E_0$ for all $f\in\A$. Since $\p_t V_t = -2 \mathcal E_t$, such an inequality is equivalent to $V_t \leq e^{-\frac{2t}{c}}V_0$, namely to an exponential decay in $L^2(\mu)$. The same goes for entropy and Gross log Sobolev inequality, or general $\Phi$-entropies (see \cite{Chafai2004} and Section \ref{SectionVariance} for some definitions), at least for diffusion processes.

For reversible processes (\emph{i.e.} when $L$ is symmetric in $L^2(\mu)$) there is a strong link between, on the one hand, Wasserstein distances and coupling and, on the other hand, variance (or entropy) and functional inequalities (see \cite{Cattiaux2008,CattiauxGuillinPAZ,Kulik}); nevertheless PDMP are not reversible. Furthermore their invariant measures usually do not satisfy a Poincaré inequality for $\Gamma$, which is non-local, not easy to handle, satisfying no chain rule (nevertheless, see \cite{Mischler2010} for a case in which such an inequality does indeed hold).

However, they may satisfy a diffusion-like Poincar\'e inequality of the form 
\begin{eqnarray}\label{EqPoincareDiffusion}
\forall f\in\A\hspace{20pt}\int f^2 d\mu - \po \int f d\mu \pf^2 & \leq & c\int |\nabla f|^2 d\mu, 
\end{eqnarray}
in other words $V_t \leq c W_t$. Such an inequality, which involves the classical energy rather than the markovian one, implies concentration properties for the measure $\mu$ (see \cite{Logsob}), but is \emph{a priori} not directly linked to the convergence to equilibirum in general.

 Suppose such an inequality holds. Then, from inequality \eqref{EqWtEvolution}, if $\eta >0$,
\begin{eqnarray*}
\p_t \po W_t + \beta V_t \pf & \leq & - \eta W_t\\
& \leq & -\frac{\eta}{1+\beta c}\po W_t + \beta V_t \pf.
\end{eqnarray*}
This yields:

\begin{thm}\label{ThmConcergenceH1Paspoids}
 Assume the Poincar\'e inequality \eqref{EqPoincareDiffusion} holds, and $|\na Q f(x)|^2 \leq M(x) Q|\na f|^2(x)$ with $M$ such that for $\mu$-almost all $x\in \Omega$ and for all $u\in\R^d$,
 \begin{eqnarray}\label{EquaBalancePasConstant}
 u^* \po 2J_b(x) + \frac{\na \lambda(x) (\na \lambda(x))^*}{\beta \lambda(x)}\pf u +\lambda(x)\po M(x)-1\pf| u |^2  & \leq &  - \eta| u |^2 
 \end{eqnarray}
 for some constants $\eta,\beta>0$. Then
 \[W_t + \beta V_t \leq \po W_0 + \beta V_0\pf e^{-\frac{\eta t}{\beta c+1}}.\]
\end{thm}
Note that
\[W_t + \beta V_t = \| \na P_t f\|^2_{L^2(\mu)} + \beta \| P_t f - \mu f \|^2_{L^2(\mu)} \]
is equivalent to the square of the usual Sobolev $\mathcal H^1$-norm of $P_t f - \mu f$. Thus Theorem \ref{ThmConcergenceH1Paspoids} provides a decay in $\mathcal H^1(\mu)$ rather than in $L^2(\mu)$. In this sense, our method can be seen as an hypocoercive method of modified Lyapunov functional (see \cite{Villani2009,DMS2011,Baudoin}, etc.), although it is quite simple. In these settings, it is usual to assume a Poincar\'e inequality \eqref{EqPoincareDiffusion} holds. There are classical criteria on a function $F$ on $\R^d$ to decide wether the law $e^{-F(x)}dx$ satisfies such an inequality, and several ways to estimate the constant $c$. However, for PDMP, the invariant law is usually quite unknown. The second part of this work will thus be dedicated to the obtention of such inequalities, which are interesting by themselves as they provide concentration bounds for the measure $\mu$.

\bigskip

The original motivation of the present work was the study of the so-called TCP process on $\Omega = \R_+$, whose generator is
\begin{eqnarray}\label{GeneTCPlin}
\forall x>0, \ f\in\A,\hspace{20pt}L f(x) &=& f'(x) + x \po f(\delta x) - f(x)\pf,
\end{eqnarray}
for some $\delta \in(0,1)$. It has been studied in \cite{ChafaiMalrieuParoux}, which inspired the main ideas of this work. In addition to the previous difficulties (no Poincar\'e inequality for $\Gamma$, non-constant rate of jump), there is another one which is particular to this process : the jump vanishes at the origin. Nevertheless, as an illustration of the efficiency of our method, we will prove the following:
\begin{prop}\label{PropTCPLinea}
For $f\in\A$, define
\[ \text{Ent} f = \mu \po f^2 \log f^2\pf - \po \mu f^2\pf \log \po \mu f^2\pf.\]
Then if $(P_t)_{t>0}$ is the semi-group associated to the generator \eqref{GeneTCPlin}, there exists $c,r>0$ such that for all $f\in\A$,
\begin{eqnarray*}
\text{Ent} P_t f & \leq & c e^{-r t} \mu (f')^2.
\end{eqnarray*}
Moreover it is possible to get explicit values for $c$ and $r$ such that this holds.
\end{prop}

\bigskip

The paper is organized as follow. Slightly generalized versions of Theorems \ref{ThmConvergencePonctuellePasPoids} and  \ref{ThmConcergenceH1Paspoids} are stated and proved in Section \ref{SectionCVH1}. A general strategy to obtain some functional inequalities (including the Poincar\'e inequality) for PDMP by the study of their embedded chain is exposed in Section \ref{SectionVariance} and applied in several illustrative models in Section \ref{SectionExampleVariance}, where in particular Proposition \ref{PropTCPLinea} is proved. A perturbative results for Poincar\'e and log-Sobolev inequalities is stated and proved in an Appendix.

\section{Exponential decay}\label{SectionCVH1}
We keep the notations and assumptions of the introduction. In particular we study the semi-group $(P_t)_{t\geq 0}$ with generator $L$ defined by \eqref{DefinitionGenerateur}.

When $A$ is a linear operator on $\A$ and $\phi$ is a bilinear symmetric one, for $f,g\in\mathcal A$ we define
\[\Gamma_{A,\phi}( f,g) = \frac12 \po A \phi(f,g) - \phi(f,Ag) - \phi(Af,g)\pf.\]
 With respect to $f$, $\Gamma_{A,\phi}(f,f)$ is quadratic, and linear with respect to $A$ and $\phi$. We will always note $f\mapsto \phi(f)$ the quadratic form associated to a bilinear form $f,g\mapsto \phi(f,g)$ and similarly we will always note $f,g\mapsto q(f,g)$ the symetric bilinear form associated by polarization to a quadratic form $f\mapsto q(f)$ on $\A$.
Let
\[\psi(s) = P_s \phi\po P_{t-s} f\pf,\hspace{15pt}s\in[0,t]\]
which interpolates between $\phi \po P_t f\pf$ and $P_t \po \phi f\pf$. Then
\[\psi'(s) = 2P_s \Gamma_{L,\phi}\po P_{t-s} f\pf.\]
To prove Theorems \ref{ThmConvergencePonctuellePasPoids} and \ref{ThmConcergenceH1Paspoids} we should consider $\phi(f) = |\nabla f|^2$. In fact it will be convenient for the applications to work with a weighted gradient $\phi_a(f) = a|\na f|^2$ with $a>0$ a scalar field on $\Omega$ in $\A$ (so that $f\in\A \Rightarrow \phi_a(f)\in\A$).


\begin{lem}\label{LemCalculGamma}
\begin{enumerate}
\item  For all $f\in\A$
\[\Gamma_{b^*\nabla,\phi_a}(f) = \frac{b^*\na a}{2a}\phi_a (f) - a(\na f)^* J_b \na f. \]
\item Suppose there exists a function $M$ on $\Omega$ such that, for all $f\in\A$, $\phi_a ( Qf) \leq M Q\po \phi_a(f)\pf $, and let $I$ be the identity operator on $\A$. Then for all $f\in\A$
\[\Gamma_{\lambda(Q-I),\phi_a}(f) \geq -a(\na f)^*(\na \lambda)(Qf - f) + \frac\lambda2(1-M)\phi_a (f) . \]
\end{enumerate}
\end{lem}

\begin{proof}
First we note that
\begin{eqnarray*}
 \na\po b^* \na f\pf & = & J_b \na f +  H_f b
\end{eqnarray*}
with $H_f(x) = \co \p_i \p_k f (x)\cf_{1\leq i,k \leq d}$ the Hessian of $f$, and
\begin{eqnarray*}
b^* \na\po a | \na f |^2\pf & = & (b^* \na a)|\na f|^2  + 2 a b^* H_f \na f
\end{eqnarray*}
Thus
\begin{eqnarray*}
\Gamma_{b^*\nabla,\phi_a}(f) & = & \frac12 b^* \na\po a | \na f |^2\pf - a(\na f)^* \na\po b^* \na f\pf\\
& = & \frac12(b^* \na a)|\na f|^2 - a (\na f)^*J_b \na f.
\end{eqnarray*}
As far as the second point is concerned,
\begin{eqnarray*}
\Gamma_{\lambda(Q-I),\phi_a}(f) & = & \frac12 \lambda \po Q\po \phi_a(f)\pf - \phi_a(f)\pf - a(\na f)^*(\na\lambda)(Qf-f) - \lambda a(\na f)^*(\na Qf - \na f)\\
& \geq & \frac\lambda 2 \po Q\po \phi_a(f)\pf +  \phi_a(f) - 2\sqrt{\phi_a(f) \phi_a( Qf)}\pf- a(\nabla f)^*(\nabla \lambda)(Qf-f).
\end{eqnarray*}
We conclude by
\[2\sqrt{\phi_a(f) \phi_a( Qf)} \leq 2\sqrt{M \phi_a(f) Q\phi_a( f)}\leq M\phi_a(f) + Q\po \phi_a(f) \pf.\]
\end{proof}
We can now state the following :
\begin{thm}\label{ThmConvergencePonctuelle}
 Assume $\lambda$ is constant and there exist a function $M$ on $\Omega$ and a constant $\eta \in\R$ such that,  for all $f\in\A$, $\phi_a\po Qf\pf \leq M Q\po \phi_a(f)\pf$ and
 \[\forall (x,u)\in\Omega\times\R^d,\hspace{20pt}2u^*J_b(x)u + \po\lambda\po M(x)-1\pf - \frac{b^*\na a(x)}{a(x)} + \eta\pf |u|^2  \leq 0.\]
Then
 \[ \phi_a( P_t f )  \leq e^{-\eta t} P_t\po \phi_a(f)\pf.\]
\end{thm}
In particular with $a=1$ we retrieve Theorem \ref{ThmConvergencePonctuellePasPoids}.
\begin{proof}
From Lemma \ref{LemCalculGamma}, since in the constant rate case $ \na \lambda = 0$,
\begin{eqnarray*}
\Gamma_{L,\phi_a}(f) & \geq & - a(\na f)^* J_b \na f + a\po \frac{b^*\na a}{2a}  + \frac\lambda2(1-M)\pf |\na f|^2\\
& \geq & \frac\eta 2 \phi_a(f).
\end{eqnarray*}
Hence if $\psi(s) = P_s\phi_a(P_{t-s} f)$,
\[\psi'(s) = 2P_s\Gamma_{L,\phi_a}(P_{t-s} f) \geq \eta \psi(s) \]
and $\psi(t) \geq e^{\eta t} \psi(0)$, which concludes.
\end{proof}
Remark that we did note use the ergodicity of the process here, and that $\eta$ can be negative.

\bigskip

This commutation between the semigroup and the gradient leads to a contraction in Wasserstein distance. More precisely, define on $\Omega$ the distance associated to the weighted gradient $D=\sqrt{ a}\na$ by
\[d(x,y) =\inf\left\{ \int_0^1 \frac{|\gamma'(s)|}{\sqrt{a\po\gamma(s)\pf}}ds ,\ \gamma : [0,1]\rightarrow \Omega,\ smooth,\  \gamma(0)=x, \gamma(t) = y\right\}\]
and the associated Wasserstein distance between two probability laws $\nu_1,\nu_2$ having a finite $p^{th}$ moment (\emph{i.e.} for which there exists a $x_0\in\Omega$ with $\nu_i\co d^p(.,x_0) \cf <\infty$) by
\[\mathcal W_{d,p}(\nu_1,\nu_2) = \underset{X\sim \nu_1,\ Y\sim \nu_2}{\inf}\po \E\co d^p(X,Y)\cf\pf^\frac1p.\]
A function $f$ will be called $\kappa$-Lipschitz with respect to $D$ if $\forall x,y\in\Omega$,
\[f(x)-f(y) \leq \kappa d(x,y).\]
This is equivalent for a smooth function to $\|D f\|_\infty \leq \kappa$, 
and we have the Kantorovich-Rubinstein dual representation (see \cite{VillaniOldNew})
 \[\mathcal W_{d,1}(\nu_1,\nu_2) = \sup\left\{\nu_1 f  - \nu_2 f,\ \|D f \|_\infty \leq 1\right\},\]
where we use the operator notation $\nu f = \int f d\nu$. 

\bigskip

Recall that by duality a Markov semi-group acts on the right on probability laws by
\[\po \nu_1 P_t\pf f := \nu_1 \po P_t f\pf.\]
If $P_t$ were absolutely continuous with respect to the Lebesgue measure for $t>0$ - which is not  the case for a PDMP since for all time $t$ there is a non-zero probability that the process hasn't jumped yet -  the gradient estimate of Theorem \ref{ThmConvergencePonctuelle} would yield, from \cite[Theorem 2.2]{Kuwada}, a contraction of the $\mathcal W_{d,2}$ distance :
 \[\mathcal W_{d,2}\po \nu_1 P_t,\ \nu_2 P_t\pf \leq e^{-\frac\eta2 t}\mathcal W_{d,2}\po \nu_1 ,\ \nu_2 \pf .\]
 Instead of trying to adapt Kuwada's result, since our work is more concerned about variance than Wasserstein distance, we will only state the weaker result :
\begin{cor}\label{ThmCVWasserstein}
In the settings of Theorem \ref{ThmConvergencePonctuelle}, for all laws $\nu_1$, $\nu_2$ with finite first moment, if $P_t\nu_1$ and $P_t\nu_2$ still have finite first moment,
 \[\mathcal W_{d,1}\po \nu_1 P_t,\ \nu_2 P_t\pf \leq e^{-\frac\eta2 t}\mathcal W_{d,1}\po \nu_1 ,\ \nu_2 \pf .\]
\end{cor}
\begin{proof}
 Theorem \ref{ThmConvergencePonctuelle} yields the weaker gradient estimate
 \[ \|D P_t f\|_\infty \leq e^{-\frac\eta2 t} P_t \|D f\|_\infty = e^{-\frac\eta2 t} \|D f\|_\infty.\]
 This implies the $\mathcal W_{d,1}$ decay, thanks to the Kantorovich-Rubinstein dual representation

\end{proof}
Note that the invariant measure does not intervene neither in Theorem \ref{ThmConvergencePonctuelle} nor in Corollary \ref{ThmCVWasserstein}, so that its existence and uniqueness are not necessary. Besides, on a complete space, a contraction of the Wasserstein distance would imply ergodicity, from \cite[Theorem 5.23]{Chen}.

We won't push the analysis further concerning the Wasserstein distance, but refer to the study in \cite{Malrieu2011} of the TCP process where an exponential decay is first obtained for a distance equivalent to $d(x,y) = \sqrt{|x-y|}$ and then is transposed to $d(x,y)=|x-y|^p$ via moments estimates and H\"older inequality. For further considerations on gradient-semigroup commutation, one shall consult \cite{Joulin2013,Ambrosio,Kuwada}.

\bigskip

We now turn to the non-constant jump rate case. Let $a \in\A$ be a non-negative scalar field on $\Omega$. Throughout all the text we will say a probability measure $\nu$ satisfies a weighted Poincar\'e inequality with constant $c$ and weight $a$ if for all $f\in\A$
\begin{eqnarray}\label{EquationPoincare}
\nu f^2 - \po \nu f\pf^2 & \leq & c \nu \po a |\na f|^2\pf.
\end{eqnarray}
Let $V_t = \mu \po P_t f\pf^2 - \po \mu f\pf^2$ and $W_t = \mu \phi_a(P_t f)$. Note that in the introduction $W_t$ was defined with the constant weight $a=1$, so that the following is slightly more general than Theorem~\ref{ThmConcergenceH1Paspoids}:
\begin{thm}\label{ThmConcergenceH1}
Assume that $\mu$ satisfies the weighted Poincar\'e inequality \eqref{EquationPoincare} with constant $c$ and weight $a$, that $\mu$-amost everywhere $\lambda > 0$ and that there exist a function $M$ and constants $\eta,\beta>0$ such that for $\mu$-almost all $x\in\Omega$, for all $f\in\A$ and for all $u\in\R^d$,  $\phi_a (Qf) \leq M Q\po \phi_a(f)\pf$ and
 \begin{eqnarray}\label{EquationThm1}
 u^*\po 2J_b(x) + \frac{a}{\beta \lambda(x)}\na \lambda(x) (\na \lambda(x))^* +\lambda(x)\po M(x)-1\pf - \frac{b^*\na a(x)}{a(x)} + \eta\pf u   & \leq &  0.
 \end{eqnarray}
  Then
 \[W_t + \beta V_t \leq e^{-\frac{\eta t}{\beta c+1}}\po W_0 + \beta V_0\pf ,\]
 and
 \[W_t \leq (1+\beta c)  e^{-\frac{\eta t}{\beta c+1}} W_0.\]
\end{thm}
\begin{proof}
Since $\mu$ is the invariant measure of the process, $\mu L g = 0$ for all $g\in\A$. In particular if $\phi$ is a quadratic form on $\A$, $\mu \po L \phi(f)\pf = 0$ and
\begin{eqnarray*}
\p_t \po \mu \po \phi(P_t f)\pf \pf & = & 2 \mu \po \phi(P_t f, LP_t f)\pf \\
& = & -2\mu \Gamma_{L,\phi}(P_t f).
\end{eqnarray*}
In particular
\begin{eqnarray*}
\p_t W_t & = & - 2 \mu \Gamma_{L,\phi_a}(P_t f).
\end{eqnarray*}
From Lemma \ref{LemCalculGamma}, 
\begin{eqnarray*}
\Gamma_{\lambda(Q-I),\phi_a}(f)  & \geq & -a(\na f)^*(\na \lambda)(Qf - f) + \frac\lambda2(1-M)\phi_a (f) \\
& \geq & - \frac{a^2}{2\beta\lambda }|(\na f)^*\na\lambda |^2 - \frac{\beta \lambda}{2}(Qf - f)^2+ \frac\lambda2(1-M)\phi_a (f) .
\end{eqnarray*}
Again from Lemma \ref{LemCalculGamma} and from Inequality \eqref{EquationThm1},
\begin{eqnarray*}
\Gamma_{L,\phi_a}(f) & \geq &  \frac{\eta}{2} \phi_a(f) - \frac{\beta\lambda}{2}(Qf - f)^2.
\end{eqnarray*}
On the other hand, if $\phi_2 (f) = f^2$ then $\Gamma_{L,\phi_2}$ is the usual \emph{carr\'e du champ} operator. From the Leibniz rule, $\Gamma_{b^*\na,\phi_2}f = 0$, so that
\begin{eqnarray*}
\p_t V_t & = & -2\mu \Gamma_{\lambda(Q-I),\phi_2} (P_t f) \\
& = & - \mu \lambda\po Q (P_t f)^2 + (P_t f)^2 - 2 (P_t f)(QP_t f)\pf\\
& \leq & -\mu \lambda \po QP_t f - P_t f\pf^2
\end{eqnarray*}
the last inequality being a consequence of the Cauchy-Schwartz inequality for $Q$. At the end of the day, we get
\[\p_t \po W_t + \beta V_t \pf \leq - \eta W_t\]
and, thanks to the weighted Poincar\'e inequality \eqref{EquationPoincare},
\[\p_t \po W_t + \beta V_t \pf \leq - \frac{\eta}{1+\beta c}( W_t + \beta V_t),\]
which yields the first assertion. Then
\[W_t \leq W_t + \beta V_t \leq e^{-\frac{\eta t}{\beta c+1}}(W_0+\beta V_0)  \leq (1+\beta c) e^{-\frac{\eta t}{\beta c+1}} W_0.\]
\end{proof}

Note  that $\eta$ could depend on $x$, so that the weight that intervenes in the Poincar\'e inequality may be different from $a$. For instance for the TCP with linear rate on $\R_+$ (Example \ref{ExampleTCPLInear}), one could consider $a(x) = x$ and $\eta(x) = - \kappa - \alpha x$ for some $\kappa,\alpha>0$. Then it would be sufficient to prove an inequality with weight $\tilde a(x) = 1+x$, which is weaker than both the classical inequality with constant weight and the inequality with weight $a$.

\section{Functional inequalities for PDMP}\label{SectionVariance}

This section is devoted to the obtention of the Poincar\'e inequality \eqref{EquationPoincare} and of slightly more general functional inequalities for $\mu$ the invariant measure of the process $(X_t)_{t\geq 0}$ with generator \eqref{DefinitionGenerateur}.

\subsection{Confining operators}

The variance is a way among others to quantify the distance to equilibrium. 
In this section we suppose that for all $f\in\A$ the so-called $p$-entropies
\begin{eqnarray*}
 \ent f & = &  \frac{\mu f^2 - \po\mu f^{\frac2p}\pf^p}{p-1} \hspace{20pt} \text{ for }p\in(1,2],\\
\text{Ent}_1 f & = & \mu \po f^2 \log f^2\pf - \po \mu f^2\pf \log \po \mu f^2\pf \end{eqnarray*}
are well-defined. We say that $\mu$ satisfies a Beckner's inequality $\mathcal B(p,c)$ if
\begin{eqnarray}\label{InegaliteBeckner}
\forall f\in\A,\hspace{20pt}\ent f  & \leq  & c\mu |\na f|^2.
\end{eqnarray}
For $p=2$ this is the Poincar\'e inequality, for $p=1$ this is the Gross log Sobolev one. Since $\ent f$ is non increasing with $p \in (1,2]$ (see \cite{Latala}; note that we took the definitions of \cite{BolleyGentil}), $\mathcal B(p,c)$ implies $\mathcal B\po q,c\pf$ whenever $q\geq p$. On the other hand by Jensen inequality $(1-p) \ent f$ is non decreasing with $p \in [1,2]$. In particular all Beckner's for $p\in(1,2]$ are equivalent up to some factor. For the global study of this inequalities and of more general $\Phi$-entropies, we refer to \cite{Chafai2004} and \cite{BolleyGentil}.

For $\alpha\in[0,1]$ we say $\mu$ satisfies a generalized Poincar\'e inequality $\mathcal I(\alpha,c)$ if
\begin{eqnarray}\label{InegaliteBeckner}
\forall f\in\A,\ \forall p\in(1,2],\hspace{20pt}(1-p)^{1-\alpha}\ent f  & \leq  & c \mu |\na f|^2.
\end{eqnarray}
For $\alpha=0$ this is still the Poincar\'e inequality, for $\alpha=1$ this is the log Sobolev one, and for $\alpha\in(0,1)$ this is an interpolation between these two cases which implies the following concentration property: there exists a constant $L>0$ such that for any borellian set $A$ with $\mu(A)\geq \frac12$, if $A_t$ is the set of points at distance at most $t$ from $A$, then $\mu(A_t)\geq 1 - e^{L t^{\frac{2}{2-\alpha}}}$ (see~\cite{Latala}). To prove $\mathcal I(\alpha,c)$ is equivalent to prove $\mathcal B\po p,c(1-p)^{\alpha-1}\pf$ for all $p\in(1,2)$.

\bigskip

In this section, for the sake of simplicity, we won't consider weighted inequalities such as the weighted Poincar\'e inequality \eqref{EquationPoincare} with $a\neq 1$. Everything would work the same, and, at least in dimension one, a weighted inequality can be seen as a non-weighted one through a change of variable (see an application in Section \ref{ExampleTCPLInear}). 

Remark that if $\mu$ satisfies $\mathcal B(p,c)$ for $p\in[1,2]$, then it satisfies a Poincar\'e inequality. In this case, providing the inequality \eqref{EquationThm1} of Theorem \ref{ThmConcergenceH1} holds, $W_t$ decays exponentialy fast, and 
\[\ent P_t f\hspace{15pt} \leq\hspace{15pt} c W_t \hspace{15pt}\leq\hspace{15pt} c(1+\beta c) e^{-\frac{\eta t}{1+\beta c}} W_0.\]


Let $\psi:\Omega\rightarrow \Omega$ be a smooth function with Jacobian matrix $J_{\psi}$, and let $|J_\psi|$ be the euclidian operator norm of $J_\psi$, namely
\[| J_\psi | = \sup \left\{ |J_\psi u|,\ u\in\R^d,\ |u|=1\right\}.\]
We say $\psi$ is $\gamma$-Lipschitz (where $\gamma \in\R_+$) if for all $x\in\Omega $, $|J_\psi(x) | \leq \gamma$. It is clear that in this case when the law of a random variable $Z$ satisfies $\mathcal B(p,c)$ then the law of $\psi(Z)$ satisfies $\mathcal B(p,\gamma^2 c)$. 
 In order to get Beckner's inequalities for the invariant law of a PDMP we will prove a generalization of this fact, based on an initial idea of Malrieu and Talay \cite{MalrieuTalay}.

Let $H$ be a Markov kernel on $\Omega$ that fixes $\A$.
\begin{defi}\label{DefContractifEntropie}
 Let $c,\gamma >0$, $p\in[1,2]$. We say that $H$ is $(c,\gamma,p)$-confining if both the following conditions are satisfied :
 \begin{itemize}
  \item \emph{sub-commutation}: $\forall f\in\A,\ \forall x\in \Omega$, 
  \begin{eqnarray}\label{EquaConfineSubComm}
   \left|\na \po H f^\frac2p\pf^\frac p2\right|^2(x)  & \leq & \gamma H|\na f|^2(x).
  \end{eqnarray}
  \item \emph{Local Beckner's inequality:}  $\forall f\in\A,\ \forall x\in \Omega,$, 
  \begin{eqnarray}\label{EquaConfineLocalB}
  \frac{H f^2(x) - \po H f^{\frac2p}\pf^p(x)}{p-1} & \leq & c H |\na f|^2(x).
  \end{eqnarray}
  if $p>1$ and
\begin{eqnarray}\label{EquaConfineLocalB1}
  H \po f^2 \ln f^2 \pf (x) -  H f^2(x)\ln H f^2 (x) & \leq & c H |\na f|^2(x).
  \end{eqnarray}
  if $p=1$.
 \end{itemize}
If $\gamma <1$ we say $H$ is $(c,\gamma,p)$-contractive. When there is no ambiguity for $p$, $H$ will simply be called confining (or contractive) if there exist $c,\gamma >0$ satisfying both conditions.
\end{defi}
 Note that \eqref{EquaConfineLocalB1} holds iff \eqref{EquaConfineLocalB} holds for all $p>1$

\bigskip

\textbf{Exemples}: \begin{itemize}
                   \item Let $\psi$ be a $\gamma$-Lipschitz function and $H f(x) = f(\psi(x))$. The sub-commutation \eqref{EquaConfineSubComm} is clear, and the local  inequality \eqref{EquaConfineLocalB} holds with $c=0$, since $H(x)$ is a Dirac mass.
                  
                   \item The sub-commutation is always satisfied with $\gamma = 0$ if $H(x) = \nu$ is a constant kernel, namely is a probability on $\Omega$, so that $\nu$ is confining iff it satisfies a Beckner's inequality.
                   \item if $N$ is a standard Gaussian vector on $\R^d$ and $(B_t)_{t\geq0}$ a Brownian motion on $\R^d$ then
                   \[K_t f(x) = \E \po f(x+B_t)\pf = \E \po f(x+\sqrt tN)\pf\]
                   is $(t,1)$-confined for the usual gradient and $p=1$ (see \cite[Chapter 1]{Logsob}). If the Brownian motion is replaced by an elliptic diffusion, a sub-commutation is given by its Bakry-Emery curvature (see \cite[Chapter 5]{Logsob}). 
                   \item Remark this definition could be extended to a Markov kernel $H:\Omega_1 \rightarrow \mathcal P\po \Omega_2\pf$ with $\Omega_1\subset\R^d$ and  $\Omega_2 \subset\R^n$. For instance if $\varphi$ is the flow associated to a vector fields $b$ on $\Omega_1$ then $Hf(t) = f\po\varphi_x(t)\pf$ is a Markov kernel from $\R_+$ to $\mathcal P(\Omega_1)$, and $\p_t Hf = H\po b^*\na f\pf$.
                  \end{itemize}
 
Here is maybe our most important, although very simple result:

\begin{lem}\label{LemConfine}
For $i=1,2$, let $H_i$ be a $(c_i,\gamma_i,p)$-confining Markov kernel on $\Omega$.
 \begin{enumerate}
                                \item  Then $H_1H_2$ is a $(c_2+\gamma_2 c_1, \gamma_1\gamma_2,p)$-confining Markov kernel.
                                \item  If  $\nu\in\mathcal P\po \Omega\pf$ satisfies $\mathcal B(p,c)$ then $\nu H_2$ satisfies $\mathcal B(p,c_2+\gamma_2 c)$.
                                \item Suppose $H$ is $(c,\gamma,p)$-contractive and the Markov chain generated by $H$ is ergodic in the sense there exists $\nu\in\mathcal P\po\Omega\pf$ such that for all $x\in \Omega$ and $f\in \A$, $H^n f(x)$ goes to $\nu f$ as  $n$ goes to infinity. Then the invariant law $\nu$ satisfies $\mathcal B\po p, c (1-\gamma)^{-1}\pf$.
                               \end{enumerate}
\end{lem}

\begin{proof}
  Let $p\in (1,2]$ (the case $p=1$ is similar and already treated in \cite{ChafaiMalrieuParoux}). First,
\[ \left|\na \po H_1H_2 f^\frac 2p\pf^\frac p 2\right|^2 \leq \gamma_1 H_1\po \left|\na\po H_2 f^\frac2p\pf^\frac p 2\right|^2\pf \leq \gamma_1\gamma_2 H_1H_2|\na f|^2\]
and 
 \begin{eqnarray*}
 \frac{H_1H_2 f^2 - \po H_1H_2 f^{\frac2p}\pf^p}{p-1}  & = & \frac{1}{p-1} \po H_1\co H_2 f^2 - \po H_2 f^\frac2p\pf^p\cf + H_1\po H_2 f^\frac2p\pf^p - \po H_1H_2 f^\frac2p\pf^p\pf\\
 & \leq & c_2 H_1H_2 |\na f|^2 + \frac{1}{p-1} \po H_1 g^2 - \po H_1 g^\frac2p\pf^p\pf\hspace{20pt}\text{ with }g =  \po H_2 f^\frac2p\pf^\frac p2\\
 & \leq & c_2 H_1H_2 |\na f|^2 + c_1 H_1 |\na g|^2\\
 & \leq & (c_2+ \gamma_2 c_1)H_1H_2 |\na f|^2.
\end{eqnarray*}
The second point is obtained from the first one by considering $H_1 = \nu$. Concerning the third assertion, by induction from the first one we get for all $n\in\mathbb N$
\begin{eqnarray*}
  \frac{H^n f^2 - \po H^n f^{\frac2p}\pf^p}{p-1} & \leq & c\po\sum_{k=0}^n \gamma^k\pf H^n |\na f|^2.
  \end{eqnarray*}
  The weak convergence of $H^n$ to $\nu$ concludes.
\end{proof}

\textbf{Example:} Let $(E_k)_{k\geq 0}$ be an i.i.d. sequence of standard exponential variables, and $(X_k)_{k\geq 0}$ be the Markov chain on $\R_+$ defined by $X_{k+1} = \frac{X_k + E_k}{2}$. Its transition operator is
\[Pf(x) \ = \ \E\po f\po \frac{x+E_0}{2}\pf\pf.\]
Clearly $(Pf)'(x) = \frac12 P(f')(x)$, so that $|\po Pf\pf'|^2 \leq \frac14 P|f'|^2$. On the other hand $P(x)$, the law of $\frac{x+E}{2}$, is the image by a $\frac12$-Lipschitz transformation of the exponential law $\mathcal E(1)$, which satisfies a Poincar\'e inequality $\mathcal B(2,4)$  (cf. Theorem \cite[Theorem 6.2.2]{Logsob} for instance). Thus $P$ is $(2,\frac14,2)$-contractive. On the other hand it is clear the chain is irreducible, it admits $C=[0,3]$ as a small set and $V(x) = x+1$ as a Lyapunov function (since $PV(x) \leq \frac34 V(x) + \mathbb 1_{x<3}$) so that it is ergodic (see \cite{MeynTweedieDown} for definitions and proof). According to Lemma \ref{LemConfine}, the invariant measure satisfies a Poincar\'e inequality $\mathcal B\po 2,\frac83\pf$.

This chain can be obtained from the TCP process with constant jump rate (Section \ref{ExampleTCPconstant} below) if the process is only observed when it jumps. This is the so-called embedded chain associated to the continuous process, which we now introduce in a general framework.

\subsection{The embedded chain}\label{SectionEmbedded}

Recall $X=(X_t)_{t\geq0}$ is a process on $\Omega$ with generator given by \eqref{DefinitionGenerateur}. Let $(S_k)_{k\geq 0}$ be the jump times of $X$ and let $Z_k=X_{S_k}$. The Markov chain $(Z_k)_{k\geq 0}$ is called the embedded chain associated to $X$. 

For $s\in[S_k,S_{k+1})$, $X_s = \varphi_{Z_k}(s-S_k)$ where we recall $\varphi_x$ is the flow associated to the vector field $b$. Since
\[\frac{d}{dt}\big( f\po \varphi_x(t)\pf\big) = (b^*\na f)\po \varphi_x(t)\pf,\]
we shall say that a function $f$ is non-decreasing (resp. constant, concave, etc.) along the flow if $t\mapsto f(\varphi_x(t))$ is non-decreasing (resp. constant, etc.) for all $x\in\Omega$; in other word if $b^*\na f \geq 0$ (resp. $=0$, etc.).

Conditionally to the event $Z_k=x$, the inter-jump time $T_k=S_{k+1}-S_k$ has a density 
\[p_x(t) = \lambda\po\varphi_x(t)\pf  e^{-\int_0^t \lambda \po\varphi_x(s)\pf ds}\]
 on $\R^+$. We assume the inter-jump times are $a.s.$ finite (which is clear if $\underset{t\rightarrow\infty}\liminf \lambda\po\varphi_x(t)\pf >0$ for all $x$), and define
\[K f(x) = \int_0^{+\infty} f\po\varphi_x(t)\pf p_x(t)dt = \E\po f(\varphi_x(T_k)) | Z_k = x\pf.\]
Then $P=KQ$ is the transition operator for the chain $Z$.

Transfering properties from $X$ to $Z$, or the converse, is far from obvious. In fact it is quite easy to find counter-examples for which one is ergodic and not the other (see examples 34.28 and 34.33 of \cite{Davis}). In \cite{Costa} this problem is solved with the definition of another embedded chain by adding observation points at constant rate. That being said, in the following we won't delve into this issue, and simply assume $Z$ has a unique invariant law $\mu_e$ (which can often be proved under conditions of irreducibility, aperiodicity and existence of a Lyapunov function). In this case we can express $\mu$ from $\mu_e$:

\begin{lem}[\textbf{Theorem 34.31 of \cite{Davis}, p.123}]\label{lemMueMu}
 Assume $C = \mu_e K\po\frac1\lambda\pf = \mu_e\co\int_0^\infty e^{-\int_0^t \lambda(\varphi_x(s))ds} dt \cf < \infty$. Then
 \[\mu f = C^{-1}\mu_e K\po\frac f\lambda\pf .\]
 In other words, $\mu = \nu_e \widetilde K$ where
 \begin{eqnarray*}
  \widetilde K f(x) & = & \frac{1}{K(\frac1\lambda)(x)} K \po \frac f\lambda\pf (x)\\
  \nu_e f & = & \frac{1}{C} \mu_e \co f K\po\frac1\lambda\pf\cf.
 \end{eqnarray*}
\end{lem}
In the following we will always assume the condition $C<\infty$ holds, so that $\nu_e$ and $\widetilde K$ are well defined.

\bigskip 

Here is our plan: from Lemma \ref{LemConfine}, we may establish a Beckner's inequality for $\mu_e$ by proving the operator $P$ is contractive. By perturbative results on functional inequalities (see \cite{Chafai2004} or Appendix) this may give an inequality for $\nu_e$. Finally, again from Lemma \ref{LemConfine}, we may transfer the inequality from $\nu_e$ to $\mu$ by proving the operator $\widetilde K$ is confining.

The rest of this section will thus  enlight some general facts which will later help us (mostly in dimension 1) prove $K$ and $\widetilde K$ are confining. It is strongly inspired by the work of Chafaï, Malrieu and Paroux \cite{ChafaiMalrieuParoux}, in which a log-Sobolev inequality is proved for the invariant measure of the embedded chain of a particular PDMP, the TCP with linear rate (see Example~\ref{ExampleTCPLInear}).

\bigskip

Recall we assumed $\lambda >0$ almost everywhere, and on every fixed point of the flow, so that
\[t\mapsto \Lambda_x(t) := \int_0^t \lambda\po\varphi_x(u)\pf du\]
is invertible for all $x\in\Omega$. Moreover since we assumed the jump times are \emph{a.s.} finite, necessarily, for all $x\in\Omega$, $\Lambda_x(t) \rightarrow \infty$ as $t\rightarrow \infty$. 
Remark that
\begin{eqnarray}\label{equationLambda(phi)}
 \Lambda_{\varphi_x(s)}(t) & = &\Lambda_x(t+s) - \Lambda_x(s)
\end{eqnarray}
which yields both
\begin{eqnarray}\label{equation(p)(Lambda)}
 b(x)^*\nabla_x \po \Lambda_x(t)\pf& = &\left.\frac{d}{ds}\right|_{s=0}\po \Lambda_{\varphi_x(s)}(t)\pf \hspace{10pt}=\hspace{10pt} \lambda\po\varphi_x(t)\pf - \lambda(x)
\end{eqnarray}
and, taking $u =\Lambda_{\varphi_x(s)}(t)$ in $t+s = \Lambda_x^{-1}\po\Lambda_{\varphi_x(s)}(t)+\Lambda_x(s) \pf$,
\begin{eqnarray}\label{equationSurLambdamoinsUn}
\Lambda_{\varphi_x(s)}^{-1} \po u \pf & = & \Lambda_x^{-1}\po u+\Lambda_x(s)\pf -s.
\end{eqnarray}
 If $X_0 = x$ and if $T_x$ is the next time of jump then
 \[E=\int_0^{T_x} \lambda\po \phi_x(u)\pf du\]
  is independant from $X_0$, and has a standard exponential law. In other words $T_x \overset{dist}= \Lambda_x^{-1}(E)$, and
$T_{\varphi_x(t)} \overset{dist}= \Lambda_{\varphi_x(t)}^{-1}\po \Lambda_x(T_x)\pf$.

\begin{lem}\label{LemRestrictLipschitz}
If $\lambda$ is non-decreasing along the flow, then for all $x\in\Omega$ and $t>0$, the law of $T_{\varphi_x(t)}$ is the image of the law of $T_x$ by a 1-Lipschitz function.
\end{lem}
\begin{proof}
Let $x\in\Omega$ and $t>0$. For $s>0$ we note  $G(s) = \Lambda^{-1}_{\varphi_x(t)}\po \Lambda_x(s)\pf$, so that $T_{\varphi_x(t)} \overset{dist}= G(T_x)$. From $\frac{d}{du}\po\Lambda_x(u)\pf = \lambda\po\varphi_x(u)\pf$, we get
\begin{eqnarray*}
G'(s) & = & \frac{\lambda\po\varphi_x(s)\pf}{\lambda\po\varphi_{\varphi_x(t)}\po \Lambda^{-1}_{\varphi_x(t)}\po \Lambda_x(s)\pf\pf\pf}\\
& = & \frac{\lambda\po\varphi_x(s)\pf}{\lambda\po\varphi_{x}\po t+G(s)\pf\pf}
\end{eqnarray*}
From the relation \eqref{equationSurLambdamoinsUn} and the fact that $\Lambda_x$ (hence $\Lambda_x^{-1}$) is non-decreasing,
\[t + G(s)  =  \Lambda_x^{-1}\po \Lambda_x(s) + \Lambda_x(t) \pf \geq \Lambda_x^{-1}\po \Lambda_x(s)\pf = s.\]
Thus ${\lambda\po\varphi_x(s)\pf}\leq{\lambda\po\varphi_{x}\po t+G(s)\pf\pf}$ and $|G'(s)|\leq 1$.
\end{proof}
The assumption that the jump rate is non-decreasing along the flow is natural in several applications where the role of the jump mechanism is to counteract a deterministic trend (growth/fragmentation models for cells \cite{Bouguet}, TCP dynamics \cite{ChafaiMalrieuParoux}, etc.). In this context, the more the system is driven away by the flow, the more it is likely to jump. From a mathematical point of view, thanks to Lemma \ref{LemRestrictLipschitz}, a Beckner's inequality for the law $K(x)$ may be transfered to $K\po \varphi_x(t)\pf$ for all $t>0$.

In fact this is also true for $\widetilde K$. Let $\widetilde T_x$ be a random variable on $\R_+$ with density $\frac{e^{-\Lambda_x(t)}}{\int_0^\infty e^{-\Lambda_x(w)}dw}$, so that
\[\widetilde K f(x) = \E\co f\po\varphi_x\po\widetilde T_x\pf\pf\cf.\]

\begin{lem}\label{LemRestrictLipschitztilde}
If $\lambda$ is non-decreasing along the flow, then for all $x\in\Omega$ and $t>0$, the law of $\widetilde T_{\varphi_x(t)}$ is the image of the law of $\widetilde T_x$ by a 1-Lipschitz function.
\end{lem}
\begin{proof}
We will prove Lemma \ref{LemRestrictLipschitz} applies here. Indeed the law of $\widetilde T_{\varphi_x(t)}$ is the law of $\widetilde T_x-t$ conditionnaly to the event $\widetilde T_x>t$, exactly as the law of $ T_{\varphi_x(t)}$ is the law of $ T_x-t$ conditionnaly to the event $ T_x>t$. We need to find a jump rate which define $\widetilde T_x$ as the jump time of a Markov process.

 Let $e^{-V(s)}ds$ be a positive probability density on $\R_+$, assume $V$ is convex and let
\begin{eqnarray*}
r(t) & = & \frac{e^{-V(t)}}{\int_t^\infty e^{-V(s)}ds}.
\end{eqnarray*}
Note that $r(t) = \frac{d}{dt}\po - \ln \int_t^\infty e^{-V(s)}ds\pf$, so that
\begin{eqnarray*}
e^{-\int_0^t r(s)ds} & = & \int_t^\infty e^{-V(s)}ds.
\end{eqnarray*}
Differentiating this equality yields
\begin{eqnarray*}
r(t)e^{-\int_0^t r(s)ds} & = & e^{-V(t)}.
\end{eqnarray*}
We want to prove $r$ is non-decreasing. From the convexity of $V$,
\[r(t) = \frac{e^{-V(t)}}{\int_t^\infty e^{-V(s)}ds} = \frac{\int_t^\infty V'(s)e^{-V(s)}ds}{\int_t^\infty e^{-V(s)}ds} \geq V'(t).\]
As a consequence,
\[r'(t) = r(t)\po r(t) - V'(t)\pf \geq 0.\]
In the case of $\widetilde T_x$, if $\lambda$ is non-decreasing along the flow then $V(t)= \Lambda_x(t)-\ln\int_0^\infty e^{-\Lambda_x(w)}dw$ is convex, so that the corresponding $r$ is non-decreasing and Lemma \ref{LemRestrictLipschitz} applies. 
\end{proof}

\begin{lem}\label{LemDeriveK}
For all $f\in\A$, $x\in\Omega$,
\begin{eqnarray*}
b(x)^*\nabla\po K f\pf(x) & = & \lambda(x)K\po \frac{b^*\nabla f}{\lambda}\pf(x).
\end{eqnarray*} 
In particular if $\lambda$ is non-decreasing along the flow, $|b^*\nabla\po K f\pf| \leq K|b^*\nabla f|$.
\end{lem}
\begin{proof}
From the representation
\[K f(x) = \E\po f\po \varphi_x(T_x)\pf\pf = \E\po f\po \varphi_x\po\Lambda_x^{-1}(E) \pf\pf\pf,\]
we compute (recall $f\in\A$ is smooth and compactly supported)
\begin{eqnarray*}
  b(x)^*\na\po K f\pf(x) & = & \left.\frac{d}{ds}\right|_{s=0} \po K f\po \varphi_x(s)\pf\pf \\
 & =& \E \po \left.\frac{d}{ds}\right|_{s=0}f\co\varphi_{\varphi_x(s)}\po\Lambda_{\varphi_x(s)}^{-1} \po E \pf\pf\cf\pf\\
  & =& \E \po \left.\frac{d}{ds}\right|_{s=0}f\co\varphi_{x}\po s+\Lambda_{\varphi_x(s)}^{-1} \po E \pf\pf\cf\pf\\
 & = &  \E \po \left.\frac{d}{ds}\right|_{s=0}f\co\varphi_x\po\Lambda_x^{-1}\po E+\Lambda_x(s)\pf\pf\cf\pf\hspace{15pt}\text{(from Relation \eqref{equationSurLambdamoinsUn})}\\
 & = & \E \po\Lambda_x'(0)\po\Lambda_x^{-1}\pf' \po E \pf (b^*\nabla f)\co\varphi_x\po\Lambda_x^{-1} \po E \pf\pf\cf\pf\\
 & = & \E \po\frac{\lambda(x)}{\lambda\po\varphi_x\po\Lambda_x^{-1} \po E \pf\pf\pf} (b^*\nabla f)\co\varphi_x\po\Lambda_x^{-1} \po E \pf\pf\cf\pf.
\end{eqnarray*}
If $\lambda$ is non-decreasing along the flow, $\lambda\po\varphi_x(t)\pf \geq \lambda(x)$ for all $t\geq 0$.
\end{proof}

\begin{lem}\label{lemKtilde}
Let $h(x) = \int_0^\infty e^{-\Lambda_x(u)}du$. Then for all $f\in\A$, $x\in\Omega$,
\[b(x)^*\na\po \widetilde K f\pf(x) = \frac{\widetilde K(h b^*\na f)(x)}{h(x)}.\]
In particular if $\lambda$ is non-decreasing along the flow $|b^*\na \widetilde K f|(x) \leq \widetilde K|b^*\na f|(x)$.
\end{lem}
\begin{proof}
\[\widetilde K f(x) = \E\co f\po\varphi_x\po\widetilde T_x\pf\pf\cf.\]
 Note that $F_x(t) = \int_0^t \frac{e^{-\Lambda_x(s)}}{\int_0^\infty e^{-\Lambda_x(w)}dw}ds$ the cumulative function of $\widetilde T_x$ is invertible. Let $U$ be a uniform random variable on $[0,1]$. Then
 \begin{eqnarray*}
  \widetilde K f(x)& =& \E\co f\po\varphi_x\po F_x^{-1}(U)\pf\pf\cf\\
  \Rightarrow\hspace{10pt} b(x)^*\nabla\widetilde K f(x) & =& \E\co \left.\frac{d}{ds}\right|_{s=0} f\po\varphi_x\po s+ F_{\varphi_x(s)}^{-1}(U)\pf\pf\cf\\
  & = & \E\co \po 1+ b(x)^*\nabla_x \po F_{x}^{-1}(U)\pf\pf (b^*\nabla f)\po \varphi_x\po F_{x}^{-1}(U)\pf\pf\cf,
 \end{eqnarray*}
 If $u\in[0,1]$, from $\nabla_x \po F_x \po F_x^{-1} (u) \pf\pf = \nabla_x  (u) =0$ we get
 \begin{eqnarray}\label{eqFx-1}
  b(x)^*\na_x \po F_x^{-1}(u) \pf & = & \frac{-b(x)^*\na_x(F_x)\po F_x^{-1}(u)\pf}{F_x'\po F_x^{-1}(u)\pf}.
 \end{eqnarray}
On the first hand $F_x'(t) = \frac{e^{-\Lambda_x(t)}}{\int_0^\infty e^{-\Lambda_x(w)}dw}$. On the other hand from Equality \eqref{equation(p)(Lambda)} we compute
\begin{eqnarray*}
 b(x)^*\na_x(F_x)(t) & =& \frac{\int_0^t\po\lambda(x)-\lambda\po\phi_x(s)\pf\pf e^{-\Lambda_x(s)}ds}{\int_0^\infty e^{-\Lambda_x(w)}dw} + F_x(t) \frac{\int_0^\infty\po\lambda\po\phi_x(w)\pf-\lambda(x)\pf e^{-\Lambda_x(w)}dw}{\int_0^\infty e^{-\Lambda_x(w)}dw}\\
 & = & \lambda(x) F_x(t) + \frac{\left[e^{-\Lambda_x(s)}\right]^t_0}{\int_0^\infty e^{-\Lambda_x(w)}dw}-\lambda(x) F_x(t) -F_x(t)\frac{\left[e^{-\Lambda_x(\omega)}\right]^\infty_0}{\int_0^\infty e^{-\Lambda_x(w)}dw}\\
 & = & \frac{-1 + e^{-\Lambda_x(t)}+F_x(t)}{\int_0^\infty e^{-\Lambda_x(w)}dw}.
\end{eqnarray*}
Relation \eqref{eqFx-1} yields
\begin{eqnarray*}
1+ b(x)^*\na_x \po F_x^{-1}(u) \pf & = & 1 -\frac{-1 + e^{-\Lambda_x\po F_x^{-1}(u) \pf}+F_x\po F_x^{-1}(u) \pf}{e^{-\Lambda_x\po F_x^{-1}(u) \pf}}\\
& = & e^{\Lambda_x\po t \pf}(1-F_x\po t\pf).
\end{eqnarray*}
with $t = F_x^{-1}(u)$. Thanks to Equation \eqref{equationLambda(phi)},
\begin{eqnarray*}
 e^{\Lambda_x\po t \pf}(1-F_x\po t\pf) & = & e^{\Lambda_x\po t \pf}\int_t^\infty \frac{e^{-\Lambda_x(s)}}{\int_0^\infty e^{-\Lambda_x(v)}dv}ds\\
 & = & \frac{{\int_0^\infty e^{-\Lambda_x(w+t) + \Lambda_x(t)}dw}}{\int_0^\infty e^{-\Lambda_x(v)}dv}\\
 & = & \frac{{\int_0^\infty e^{-\Lambda_{\varphi_x(t)}(w)}dw}}{\int_0^\infty e^{-\Lambda_x(v)}dv}.
\end{eqnarray*}
Bringing the pieces together, we have proved
\begin{eqnarray*}
b(x)^*\nabla\po\widetilde K f\pf(x)  & = & \E\co \frac{h\po\varphi_x\po F_x^{-1}(U)\pf\pf}{h(x)}(b^*\nabla f)\po \varphi_x\po F_{x}^{-1}(U)\pf\pf\cf\\
& = & \frac{\widetilde K\po h b^*\na f\pf(x)}{h(x)}
\end{eqnarray*}
When $\lambda$ is non-decreasing along the flow, from \eqref{equation(p)(Lambda)}, $x\mapsto \Lambda_x(t)$ is non-decreasing along the flow for all $t\geq0$, and $h\po\varphi_x(t)\pf \leq h(x)$.
\end{proof}

\section{Examples}\label{SectionExampleVariance}


\subsection{The TCP with constant rate}\label{ExampleTCPconstant}

A simple yet instructive example on $\R_+$  is the TCP with constant rate of jump with generator
\[Lf(x) = f'(x) + \lambda \po \E\po f(Rx)\pf - f(x) \pf\]
where $R$ is a random variable on $[0,1)$ and $\lambda>0$ is constant. It is a simple growth/fragmentation model, or may be obtained by renormalizing a pure fragmentation model (cf. \cite{GabrielSalvarani} for instance). In \cite{LopkerVL,Last}, ergodicity is proved and it is shown the moments of the invariant measure $\mu$ are all finite.

Applying Theorem \ref{ThmConvergencePonctuelle} with $J_b = 0$, $M = \E\po R^2\pf$ and $a=1$, we get
\begin{prop}\label{TCPtauxconstantGradient}
for all $f\in\A$,
\[ |(P_t f)'|^2 \leq e^{-\lambda \po 1-\E\po R^2\pf\pf t }P_t |f'|^2. \]
\end{prop}
Corollary \ref{ThmCVWasserstein} then yields a contraction at rate $\lambda \po 1-\E\po R^2\pf\pf$ of the Wasserstein distance $\mathcal W_{1}(\nu_1P_t,\nu_2P_t)$. In fact by coupling two processes starting at different points to have the same jump times and the same factor $R$ at each jump, one get that for any $p\geq 1$, the $\mathcal W_{p}$ distance decays at rate $\lambda p^{-1} \po 1-\E\po R^p\pf\pf$ (see \cite{ChafaiMalrieuParoux}), and those rates are optimal (see \cite{MalrieuPDMP}). In particular $\lambda \po 1-\E\po R^2\pf\pf$ is the rate of decay of $\mathcal W_{2}^2$. 

\bigskip

Let
\[K f = \int_0^\infty f(x+s) \lambda e^{-\lambda s} ds.\]
Obviously $(Kf)' = K(f')$. Moreover the exponential law $\mathcal E(1)$ satisfies a Poincar\'e inequality $\mathcal B(2,4)$, so that by the change of variable $z \mapsto z/\lambda$,  $\mathcal E(\lambda)$ satisfies  $\mathcal B(2,4\lambda^{-2})$. Finally, the law $K(x)$ is the image of $\mathcal E(\lambda)$  by the translation $u\mapsto u +x$, which is a 1-Lipschitz transformation. As a conclusion,
\begin{lem}
  The operator $K$ is $(4\lambda^{-2},1,2)$-confining.
\end{lem}
 As far as the jump operator $Qf(x) = \E\po f(R x)\pf$ is concerned, we have already used the sub-commutation 
\[\po (Qf)'\pf^2 \leq \E\po R^2 \pf Q(f')^2.\]
However a local Poincar\'e inequality \eqref{EquaConfineLocalB} for $Q(x)$ would mean $\forall f\in\A,\ x>0$,
\begin{eqnarray*}
 \E \po f^2(R x) \pf - \po \E \po f(R x)\pf\pf^2& \leq &c \E \co \po f'(R x)\pf^2 \cf\\
 \Leftrightarrow\hspace{15pt}\E \po g_x^2(R) \pf -\po \E \po g_x(R)\pf\pf^2 &\leq &\frac{c}{x^2} \E \co \po g_x'(R)\pf^2 \cf
\end{eqnarray*}
with $g_x(r) = f(r x)$. This implies the law of $R$ satisfies $\mathcal B\po 2,cx^{-2}\pf$ for all $x>0$, hence $\mathcal B\po 2,0\pf$, which means $R$ is deterministic. Indeed, when $R$ is deterministic, the local inequality always holds:
\begin{lem}
If $R = \delta$ a.s. with a constant $\delta \in[0,1)$ then $Q$ is $(0,\delta^2,p)$-contractive.
\end{lem}
When $R$ is random, what prevent to straightforwardly use our argument is the possibility of arbitrarily little concentrated jump, for instance with uniform law on $(0,x)$ for any $x$. It's a shame because if, say, $R$ is uniform on $\po0,\frac12\pf$, it means when the process jumps it is at least divided by 2 but can be even much more contracted. In particular its invariant measure should be more concentrated near zero than the process with  $R = \frac12$ a.s. for which, as we will see, the invariant measure satisfies a Poincar\'e inequality. This illustrates a limit of our procedure.

\begin{prop}\label{PropTCPtauxconstant}
If $R = \delta$ is deterministic then $\mu$ satisfies the Poincar\'e inequality
\begin{eqnarray*}
\forall f\in\A,\hspace{20pt} \mu (f - \mu f)^2 & \leq & \frac{4}{\lambda^2(1-\delta^2)}\mu (f')^2.
\end{eqnarray*}
As a consequence,
\begin{eqnarray*}
\forall f\in\A,\hspace{20pt} \mu (P_t f - \mu f)^2  & \leq & \frac{4e^{-\lambda(1-\delta^2)t}}{\lambda(1-\delta^2)}\mu (f')^2.
\end{eqnarray*}
\end{prop}
\begin{proof}
Since $K$ and $Q$ are confining, from Lemma \ref{LemConfine}, $P=KQ$ is $(4\lambda^{-2}\delta^2,\delta^2,2)$-confining and $\mu_e$ the invariant measure of the embedded chain satisfies $\mathcal B\po 2,\frac{4\delta^2}{\lambda^2(1-\delta^2)}\pf$. From Lemma \ref{lemMueMu}, $\mu=\mu_e K$ and so by Lemma \ref{LemConfine} again $\mu$ satisfies $\mathcal B\po 2,\frac{4}{\lambda^2(1-\delta^2)}\pf$. The second inequality is a consequence of this Poincar\'e inequality and of Proposition \ref{TCPtauxconstantGradient}.
\end{proof}

In fact in this example the spectrum of the generator in $L^2(\mu)$ is explicit: there are polynomial eigenfunctions, and since the tail of $\mu$ is exponential, polynomials are dense in $L^2(\mu)$ and these eigenfunctions are the only one in $L^2(\mu)$ . The eigenvalues are $l_k=\lambda(\E\co R^k\cf-1)$ with $k\in\mathbb Z^+$. The convergence rate of the $L^2$-norm obtained in Proposition \ref{PropTCPtauxconstant} for a deterministic $R$ appears to be $\frac12|l_2|$ and not the spectral gap $|l_1|$, and of course
\[\frac12|l_2| = \lambda \E\co (1-R)\frac{1+R}2\cf \leq  \lambda \E (1-R) = |l_1|.\]
Nevertheless $\frac12|l_1|\leq \frac12|l_2|$ so we get the right rate up to a factor $1/2$.

\subsection{The storage model}

Let $U$ be a positive random variable, and consider the generator on $\R_+$
\[Lf(x) = - xf'(x) + \lambda\po \E\co f(x+U)\cf - f(x)\pf.\]
This is, in a sense, the converse of the TCP: the jumps send the process away from 0 and the flow brings it back. Applying Theorem \ref{ThmConvergencePonctuelle} with $M= 1$, $a=1$ and $J_b = -1$, we get
\begin{eqnarray}\label{InegalStorage}
|\nabla P_t f|^2 \leq e^{-2t} P_t |\nabla f|^2.
\end{eqnarray}
Besides in this case it is easy to obtain a Wasserstein decay, as the distance $s$ between two processes starting at different point and coupled to have the same jump times and the same jump sizes $U$ at each jump satisfies $s'=-s$, and such a decay implies \eqref{InegalStorage} (see \cite{Kuwada}; the converse is not clear, since $P_t$ is a mix of a Dirac mass and a smooth density).

To prove a Beckner's inequality, the same problem arises as in the previous example with a random $R$: here the law $K(x)$, namely the law of $e^{-T} x$ with $T$ an exponential random variable, can be as little concentrated as possible when $x$ goes to infinity, so that $K$ does not satisfy a local Beckner's inequality \eqref{EquaConfineLocalB}.

\subsection{The TCP with increasing rate}

Consider the generator on $\R_+$
\begin{eqnarray}\label{EqgenerateurTCP}
Lf(x) & = & f'(x) + \lambda(x) \po f(\delta x) - f(x) \pf.
\end{eqnarray}
We have already studied the constant rate case. Before tackling the case of $\lambda(x) = x$, we consider in this section an intermediate difficulty, with the following assumptions: $\lambda$ is non-decreasing, $\lambda(0)=\lambda_*>0$, and $\ln \lambda$ is a $\kappa$-Lipschitz function. Let $\beta = \frac{2\kappa^2}{1-\delta^2}$, so that
\begin{eqnarray*}
\frac{(\lambda')^2}{\beta \lambda} - \lambda\po 1 - \delta^2\pf & = & \frac{\lambda\po 1 - \delta^2\pf}2\po \frac{(\lambda')^2}{\lambda^2 \kappa^2}-2\pf\\
& \leq & -\frac{\lambda_*\po 1 - \delta^2\pf}2.
\end{eqnarray*}
In other word, Inequality \eqref{EquationThm1} holds with $\eta = -\frac{\lambda_*\po 1 - \delta^2\pf}2$ and $a=1$. To apply Theorem \ref{ThmConcergenceH1}, we also need to prove a Poincar\'e inequality.

\begin{lem}
The operators
\[K f(x) = \int_0^\infty f(x+t) \lambda(x+t) e^{-\int_0^t \lambda(x+s)ds}dt\]
and
\[\widetilde K f(x) = \int_0^\infty f(x+t) \frac{ e^{-\int_0^t \lambda(x+s)ds}}{\int_0^\infty e^{-\int_0^u \lambda(x+s)ds}}dt\]
are $\po \frac{4}{\lambda_*^2},1,2\pf$-confining.
\end{lem}
\begin{proof}
The sub-commutation \eqref{EquaConfineSubComm} is a direct consequence of Lemma \ref{LemDeriveK} and \ref{lemKtilde}, since the rate of jump is non-decreasing and $b=1$. On the other hand $K(x)$ (resp $\widetilde K(x)$) is the law of $x+T_x$ (resp. $x+\widetilde T_x$) which is from Lemma \ref{LemRestrictLipschitz} the image by a 1-Lipschitz function of $T_0$ (resp. $\widetilde T_0$). 
Thus we only need to prove the inequality holds for $K(0)$ and $\widetilde K(0)$.

For the case of $K(0)$, denote by $F(t) = 1 - e^{-\Lambda_0(t)}$ the cumulative function of $T_0$. Then, if $E$ is a standard exponential random variable,
\begin{eqnarray*}
T_0  & \overset{dist}=&  F^{-1}\po 1 - e^{- E}\pf\\
 & = & \Lambda_0^{-1}( E).
\end{eqnarray*}
Since $\Lambda_0^{-1}$ is a non-decreasing concave function with $\po\Lambda_0^{-1}\pf '(0) = \frac{1}{\lambda_*}$, $T_0$ is a $\frac{1}{\lambda_*}$-Lipschitz transformation of $E$, whose law satisfies the Poincar\'e inequality $\mathcal B(2,4)$.

In Lemma \ref{LemRestrictLipschitztilde} we saw the cumulative function of $\widetilde T_0$ is $t\mapsto 1-e^{-\int_0^t r(s)ds}$ with an increasing function $r$ defined by
\[r(t) = \frac{e^{-\Lambda_0(t)}}{\int_t^\infty e^{-\Lambda_0(s)}ds}.\]
The previous argument shows $\widetilde T_0$ is a $\frac{1}{r(0)}$-Lipschitz transformation of $E$, and
\[ r(0) = \frac{1}{\int_0^\infty e^{-\Lambda_0(s)}ds} \geq \frac{1}{\int_0^\infty e^{-\lambda_*s}ds} = \lambda_*.\]
\end{proof}
\textbf{Remark:} in fact if moreover $\lambda(x) \geq k(1+x)^q$ for some $k>0$ and $q\in[0,1]$, the laws of $T_0$ and $\widetilde T_0$ satisfy some generalized Poincar\'e inequality $\mathcal I(\alpha,c)$ with $\alpha = \frac{2q}{q+1}$ (see \cite[Theorem 3]{Barthe} and \cite{Chafai2004}), or in other words the Beckner's inequalities $\mathcal B\po p,c(1-p)^{\alpha-1}\pf$ for all $p\in(1,2]$. By the previous arguments, $K$ and $\widetilde K$ are $\po c(1-p)^{\alpha-1},1,p\pf$-confining for all $p\in(1,2]$.

\begin{cor}\label{CorTCPIncrease}
The invariant measure $\mu$ of the process satisfies a Poincaré inequality $\mathcal B(2,c)$ for some explicit $c>0$.
\end{cor}
\begin{proof}
It is clear the jump operator $Q$ is $(0,\delta^2,2)$-contractive, so that from Lemma \ref{LemConfine}, $P=KQ$ is $\po \frac{4\delta^2}{\lambda_*^2},\delta^2,2\pf$-contractive, and $\mu_e$ the invariant measure of the embedded chain associated with the process satisfies a Poincar\'e inequality $\mathcal B\po 2,\frac{4\delta^2}{\lambda_*^2(1-\delta^2)}\pf$. Let 
\[h(x) = K\po\frac1\lambda\pf(x) = \int_0^\infty e^{-\int_0^s \lambda(x+u)du}ds.\]
It is a non-increasing function with $h(0)\leq  \int_0^\infty e^{-\lambda_* s}ds = \frac1{\lambda_*}$. In order to prove the perturbation $\nu_e$ of $\mu_e$, defined by $\nu_e (f) = \frac{1}{\mu_e(h)} \mu_e (fh)$, satisfies a Poincar\'e inequality, we will use Lemma \ref{lemPerturbMonotone}, which requires an upper bound on the median $m_e$ of $\mu_e$. Note that it is possible to couple a process $X$ with rate $\lambda$ and a process $Z$ with constant rate $\lambda_*$ so that, if they start at the same point, the first one will always stay below the second one: suppose such a coupling $(X,Z)$ has been defined up to a jump time $T_k$ of $X$. Then both process increases linearily up to the next jump time $T_{k+1}$ of $X$. At time $T_{k+1}$, $X$ jumps, but $Z$ jumps only with probability  $\frac{\lambda_*}{ \lambda(X_{T_k}+T_k)}$, else it does not move. In other words the jump part of the generator of $Z$ is thought as
\[\lambda_* \po f(\delta x) - f(x)\pf = \lambda(x) \po \po \frac{\lambda_*}{\lambda(x)} f(\delta x) + \po 1-\frac{\lambda_*}{\lambda(x)}\pf f(x)\pf- f(x)\pf.\]
Such a coupling proves $m_e$ is less than the median of the invariant law of the process with constant rate $\lambda_*$. Let $Z_\infty$ be a random variable with this invariant law, so that, if $E$ is a standard exponential random variable,
\begin{eqnarray*}
Z_\infty & \overset{dist.}= & \delta \po Z_\infty + \frac1\lambda_* E\pf\\
\Rightarrow\hspace{20pt}(1-\delta)\E(Z_\infty) & = & \frac{\delta}{\lambda_*}.
\end{eqnarray*}
Hence from Markov's inequality, $m_e \leq \frac{2\delta}{\lambda_*(1-\delta)}$. 
Finally, from Lemma \ref{lemPerturbMonotone}, $\nu_e$ satisfies a Poincar\'e inequality with constant
\[c' =  \frac{32\delta^2 }{\lambda_*^3(1-\delta^2) h\po \frac{2\delta}{\lambda_*(1-\delta)}\pf}, \]
and since $\widetilde K$ is confining, from Lemma \ref{LemConfine}, $\mu = \nu_e \widetilde K$ satisfies such an inequality with constant
\[c = \frac{4\delta^2}{\lambda_*^2} + c'.\]
\end{proof}
\textbf{Remark:} if, again, $\lambda(x) \geq k(1+x)^q$ for some $k>0$ and $q\in[0,1]$, these arguments prove the invariant measure satisfies a generalized Poincar\'e inequality $\mathcal I(\alpha,c)$ for some $c>0$ and $\alpha = \frac{2q}{q+1}$. Thus the invariant measure inherits the concentration properties of the law of the jump time $T_0$: the logarithm of its density tail is (at most) of order $-x^{q+1}$.

\bigskip

Let $(P_t)_{t\geq 0}$ be the semi-group associated to the generator \eqref{EqgenerateurTCP} and for $f\in\A$ let $W_t = \mu \po (P_t f)'\pf^2$ and $V_t = \mu \po P_t f - \mu f \pf ^2$. We have proved Theorem \ref{ThmConcergenceH1} holds:

\begin{cor}
If $\lambda$ is increasing with $\lambda(0)=\lambda_*>0$ and $\ln\lambda$ is $\kappa$-Lipschitz then
\[W_t + \beta V_t \leq (W_0 + \beta V_0) e^{-\frac{\eta}{1+\beta c} t}.\]
with $c$ given by Corollary \ref{CorTCPIncrease} and
\begin{eqnarray*}
\eta & = & \frac{\lambda_*(1-\delta^2)}{2}\\
\beta  & = &  \frac{2\kappa^2}{1-\delta^2}.
\end{eqnarray*}
\end{cor}

\subsection{The TCP with linear rate}\label{ExampleTCPLInear}

In this section,
\begin{eqnarray}\label{EqgenerateurTCP}
Lf(x) & = & f'(x) + x \po f(\delta x) - f(x) \pf,
\end{eqnarray}
where $\delta\in[0,1)$, and we will prove Proposition \ref{PropTCPLinea}. We keep the general notations for $(P_t)_{t\geq 0}$, $Q$, $\lambda$ and $\mu$ (for the proof of ergodicity, see \cite{Dumas}), and write $\text{Ent} f = \mu \po f^2\ln f^2\pf - \mu (f^2) \ln \mu( f^2)$.  

\bigskip

In the first instance, from Theorem \ref{ThmConcergenceH1}, Proposition \ref{PropTCPLinea} is proven in Section \ref{Subsub1} under the additional assumption that the invariant law satisfies some weighted functional inequalities. These weighted inequalities are equivalent to non-weighted inequalities for the invariant measure of a twisted process, and the latter may be established thanks to the tools of Section \ref{SectionVariance}. More precisely, in Section \ref{Subsub2}, we prove that the transition operator of the embedded chain corresponding to the twisted process is contractive, which imply its invariant law satisfies a log-Sobolev inequality, and in Section \ref{Subsub3} we transfer this inequality to the continuous-time process via perturbative arguments.

\subsubsection{Decay of the gradient, given the weighted functional inequalities}\label{Subsub1}

Recall Theorem \ref{ThmConcergenceH1} is based on a balance condition on the way the space is contracted or expanded by the drift and the jumps. Here, the deterministic motion is just a translation at constant speed: the flow is isometric. On the other hand the jumps mechanism do contract the space, but there are few jumps in the vicinity of the origin, and thus a condition as \eqref{EquaBalancePasConstant} cannot hold uniformly in $x>0$ with $\eta >0$. An idea is to consider a metric different from the euclidian one which is uniformly contracted for all $x>0$. This metric can be equivalent to the euclidian one for $x$ away from 0, but near 0, it should distend the distances, so that the deterministic flow $\phi_x(t) = x+t$ contracts the new metric (this is reminescent of the construction of the Lyapunov function $\tilde V$ in \cite[Section 3]{Malrieu2011}) .

As we saw on Section \ref{SectionCVH1}, working with another metric is equivalent to working with weighted gradients, namely considering the condition \eqref{EquationThm1} with $a\neq 1$. To cope with the rate of jump that vanishes at the origin, we will apply Theorem \ref{ThmConcergenceH1} with a weight $a$ that behaves linearily near 0. More precisely, let 
\begin{eqnarray*}
a(x) & = &  1 - e^{-x},\\
\phi_a(f)&  = & a|f'|^2,\\
W_t & = & \mu\po \phi_a\po P_t f\pf\pf.
\end{eqnarray*}

\begin{lem}\label{LemmeTCPlinePoinca}
Suppose $\mu$ satisfies the weighted Poincar\'e inequality
\[\forall f\in\A,\hspace{20pt}\mu \po f -\mu f\pf ^2 \leq c\mu \po \phi_a(f) \pf \]
for some $c>0$, and let
\[\theta = \po \frac{3+\sqrt 5}{2} - 1\pf^{-1} + \ln\po \frac{3+\sqrt 5}{2}\pf \simeq 1.58.\]
Then for all $\beta > \po (1-\delta) \theta\pf^{-1}$, $t>0$ and $f\in\A$,
\[W_t \leq e^{- \frac{(1-\delta) \theta - \frac{1}{ \beta}}{1+\beta c}t}\po 1 + \beta c\pf W_0.\]
\end{lem}
\begin{proof}
Note that $a$ is a concave function, so that
\[a\po \delta x \pf = a\po \delta x + (1-\delta) 0\pf \geq \delta a(x) + (1-\delta)a(0) = \delta a(x). \]
Therefore
\[\phi_a(Qf)(x) = a(x)\delta^2 |f'(\delta x)|^2 \leq \delta a(\delta x)|f'(\delta x)|^2 = \delta Q\po \phi_a(f)\pf(x). \]
To apply Theorem \ref{ThmConcergenceH1} we thus have to bound below
\begin{eqnarray*}
\frac{a'(x)}{a(x)} +  x(1-\delta) - \frac{a(x)}{x \beta} & \geq & (1-\delta)\po \frac{1}{e^{x}-1} + x \pf - \frac1\beta.
\end{eqnarray*}
The function $g(x) = \frac{1}{e^{x}-1} + x $ goes to $+\infty$ at $0$ and $+\infty$ and admits a unique positive critical point for which
\begin{eqnarray*}
e^x & =& (e^x-1)^2\\
\Rightarrow \hspace{20pt}x & = & \ln\po \frac{3+\sqrt 5}2\pf.
\end{eqnarray*}
Hence for all $x>0$, $g(x) \geq g\po \ln\po \frac{3+\sqrt 5}2\pf\pf = \theta$ and Theorem \ref{ThmConcergenceH1} holds with $ \eta = (1-\delta)\theta - \frac{1}{\beta}$.
\end{proof}

\begin{cor}\label{CorTCPlineLogSob}
Suppose $\mu$ satisfies the weighted inequalities, for all $f\in\A$,
\begin{eqnarray}
\mu \po f -\mu f\pf ^2 & \leq&  c_1\mu \po \phi_a(f) \pf\notag,\\
& & \notag\\
\text{Ent}f & \leq & c_2\mu \po \phi_a(f) \pf \label{InegalLogSobPoids}
\end{eqnarray}
for some $c_1,c_2>0$, and let $\theta$ be such as defined in Lemma \ref{LemmeTCPlinePoinca}.
Then for all $\beta > \po (1-\delta) \theta\pf^{-1}$, $t>0$ and $f\in\A$,
\[\text{Ent} P_t f \leq c_2e^{- \frac{(1-\delta) \theta - \frac{1}{ \beta}}{1+\beta c_1}t}\po 1 + \beta c_1\pf \mu (f')^2.\]
\end{cor}
\begin{proof}
From Lemma \ref{LemmeTCPlinePoinca} and the fact $a \leq 1$,
\[\text{Ent}P_t f \leq c_2 W_t \leq c_2e^{- \frac{ (1-\delta)\theta - \frac{1}{ \beta}}{1+\beta c_1}t}\po 1 + \beta c_1\pf W_0 \leq c_2e^{- \frac{ (1-\delta)\theta - \frac{1}{ \beta}}{1+\beta c_1}t}\po 1 + \beta c_1\pf \mu (f')^2.\]
\end{proof}
Thus, to prove Proposition \ref{PropTCPLinea}, it only remains to prove a weighted log-Sobolev inequality holds. 
Let
\[\psi(x) = \int_0^x \frac{1}{\sqrt{a(y)}}dy.\]
It is a concave, non-decreasing, one-to-one function. If $Z$ is a random variable with law $\mu$ and $Y=\psi(Z)$, then
\begin{eqnarray*}
\E\po f^2(Z) \ln f^2(Z) \pf - \E\po f^2(Z) \pf\ln \E\po f^2(Z) \pf & \leq & c \E\po a(Z) (f')^2(Z)\pf\\
\Leftrightarrow\hspace{15pt}\E\po g^2\po Y\pf \ln g^2\po Y\pf  \pf - \E\po g^2\po Y\pf  \pf\ln \E\po g^2\po Y\pf  \pf & \leq & c \E\po (g')^2\po Y\pf \pf
\end{eqnarray*}
with $g(y) = f\po \psi^{-1}(y)\pf$. As a consequence we will study the Markov process $\psi(X) = \po\psi(X_t)\pf_{t\geq 0}$, where $X = \po X_t\pf_{t\geq 0}$ has generator \eqref{EqgenerateurTCP}, and prove a classical non-weighted log-Sobolev for the invariant measure of this twisted process, which will imply the weighted log-Sobolev assumed in Corollary \ref{CorTCPlineLogSob}.

\subsubsection{Confining operators for the twisted process}\label{Subsub2}

The jump kernel of $\psi(X)$ is
\[Q_\psi g(z) = g\po \psi\po \delta \psi^{-1}(z)\pf\pf.\]
Let $K_\psi$ and $\widetilde K_\psi$ be the operators defined in Section \ref{SectionEmbedded}  corresponding to the process $\psi(X)$.
\begin{lem}\label{LemKalphcont}
For all $g\in \A$,
\begin{eqnarray*}
|(Q_\psi g)'| & \leq & \sqrt\delta Q_\psi |g'|\\
|(K_\psi g)'| & \leq & K_\psi|g'|\\
|(\widetilde K_\psi g)'| & \leq & \widetilde K_\psi |g'|.
\end{eqnarray*}
\end{lem}
\begin{proof}
Recall $ a\po \delta x \pf \geq \delta a(x)$ for all $x\geq0$, and so
\begin{eqnarray*}
(Q_\psi g)'(z) & = & \delta \po\psi^ {-1}\pf'(z) \psi'\po \delta \psi^{-1}(z)\pf Q_\psi g'(z)\\
& = & \frac{\delta \po\psi^ {-1}\pf'(z)}{\sqrt{ a\po \delta \psi^{-1}(z)\pf}} Q_\psi g'(z)\\
& \leq & \frac{\sqrt \delta \po\psi^ {-1}\pf'(z)}{\sqrt{ a\po  \psi^{-1}(z)\pf}} Q_\psi| g'|(z)\\
 & = & \sqrt \delta Q_\psi |g'|(z).
\end{eqnarray*}
On the other hand the vector field associated to $\psi(X)$ is $b_\psi(z) = \frac{1}{\sqrt{a\po \psi^{-1}(z)\pf}}$, and the rate of jump is non-decreasing along the flow. Hence, according to Lemma \ref{LemDeriveK},
\begin{eqnarray*}
b_\psi |(K_\psi g)'| & \leq & K_\psi\po b_\psi |g'|\pf
\end{eqnarray*}
(and according to Lemma \ref{lemKtilde}, the same goes for $\widetilde K_\alpha$). Note that the support of both probability measure $K_\psi(z)$ and $\widetilde K_\psi(z)$ is $[z,\infty]$, and that $b_\psi$ is non-increasing along the flow, so that
\[|(K_\psi g)'|(z) \leq \frac{K_\psi\po b_\psi |g'|\pf (z)}{b_\psi(z)}\leq K_\psi\po  |g'|\pf(z)\]
(and the same goes for $\widetilde K_\psi$).
\end{proof}
\begin{lem}\label{LemKalphlip}
For any $z>0$, the law $K_\psi(z)$ (resp. $\widetilde K_\psi(z)$) can be obtained from $K_\psi(0)$ (resp. $\widetilde K_\psi(0)$) through a 1-Lipschitz transformation.
\end{lem}
\begin{proof}
Let $T_x$ be the first time of jump of $X$ starting from $x$. According to Lemma \ref{LemRestrictLipschitz}, there exists a 1-Lipschitz functions $G$ such that $T_x \overset{dist}= G(T_0)$. Note that $K_\psi\po \psi(x)\pf$ is the law of $\psi\po x + T_x\pf$. Let $H(z) = \psi\po x + G\po \psi^{-1}(z)\pf\pf$, so that $\psi(x+T_x) \overset{dist}= H\po \psi(T_0)\pf$. We compute
\begin{eqnarray*}
| H'(z)| & = & |G'\po \psi^{-1}(z)\pf\po\psi^{-1}\pf'(z)\psi'\po x + G\po \psi^{-1}(z)\pf\pf|\\
& \leq & \frac{\psi'\po x + G\po \psi^{-1}(z)\pf\pf}{\psi'\po\psi^{-1}(z)\pf}. 
\end{eqnarray*}
Now $\psi$ is concave, and in the proof of Lemma \ref{LemRestrictLipschitz} we have seen that $x + G(s) \geq s$ for all $s\geq 0$; hence $|H'(z)|\leq 1$ for all $z\geq 0$.

Similarly, let $\widetilde T_x$ be a random variable on $\R_+$ with density $\frac{e^{-\int_0^t (x+s) ds}}{\int_0^\infty e^{-\int_0^u (x+s) ds}du}$, so that $\widetilde K_\psi\po \psi(x)\pf$ is the law of $\psi\po x + \widetilde T_x\pf$.  From Lemma \ref{LemRestrictLipschitztilde} there exists a 1-Lipschitz functions $\widetilde G$ such that $\widetilde T_x \overset{dist}=  \widetilde G( \widetilde T_0)$, and the previous argument concludes.
\end{proof}

\begin{lem}\label{LemKalph0}
Both $K_\psi(0)$ and $\widetilde K_\psi(0)$ satisfies $\mathcal B(1,4)$.
\end{lem}
\begin{proof}
If  $T_0$ is the first time of jump starting from 0 then $K_\psi(0)$ is the law of $\psi(T_0)$. For any $f\in\A$,
\begin{eqnarray*}
K_\psi f(0) & = & \int_0^\infty f\po \psi(u)\pf u e^{-\frac{u^2}2}du\\
& = & \int_0^\infty f\po z\pf  e^{-\frac{\po\psi^{-1}(z)\pf^2}2 + \ln \psi^{-1}(z) + \frac12 \ln\po a\po \psi^{-1}(z)\pf \pf}dz.
\end{eqnarray*}
On the other hand, if $N$ is a standard gaussian variable then $\widetilde K_\psi(0)$ is the law of $\psi\po|N|\pf$, and for all $f\in\A$
\begin{eqnarray*}
\widetilde  K_\psi f(0) & = & \int_0^\infty f\po \psi(u)\pf \po\frac\pi2\pf^{-\frac12}e^{-\frac{u^2}2}du\\
& = & \int_0^\infty f\po z\pf  e^{-\frac{\po\psi^{-1}(z)\pf^2}2 + \frac12 \ln\po a\po \psi^{-1}(z)\pf \pf - \frac12 \ln\po\frac\pi2\pf}dz.
\end{eqnarray*}
For $\varepsilon\in\{0,1\}$, let $V_\varepsilon(z) = \frac12\po\psi^{-1}(z)\pf^2 - \varepsilon \ln \psi^{-1}(z) - \frac12 \ln\po a\po \psi^{-1}(z)\pf \pf$; we want to prove $V_\varepsilon$ is strictly convex. Writing $x = \psi^{-1}(z)$, we compute $\partial_z(x) = \sqrt{ a(x)}$ and
\begin{eqnarray*}
V'_\varepsilon(z) & = & \sqrt{a(x)}\po x - \frac{\varepsilon}{x} - \frac{a'(x)}{2a(x)}\pf\\
V_\varepsilon''(z) & = & \frac{a'(x)}{2}\po x - \frac{\varepsilon}{x} - \frac{a'(x)}{2a(x)}\pf + a(x) \po 1 + \frac{\varepsilon}{x^2} - \frac{ a''(x)}{2 a(x)} + \frac12\po\frac{ a'(x)}{a(x)}\pf^2\pf \\
& = & \varepsilon\po \frac{ a(x)}{x^2}-\frac{ a'(x)}{2x}\pf  + \frac{a'(x)x}{2} + \frac{\po a'(x)\pf^2}{4 a(x)} +  a(x) - \frac12 a''(x).
\end{eqnarray*}
As a first step, note that $V_1''(z) \geq V_0''(z)$: indeed, $V_1''(z)-V_0''(z)= \frac{j(x)}{x^2}$ with
\begin{eqnarray*}
j(y) & = & a(y) - \frac{y}{2}a'(y)\\
\Rightarrow\hspace{15pt} j'(y) & = & \frac12 a'(y) - \frac y2 a''(y) \ > \ 0
\end{eqnarray*}
(since $a$ is non-decreasing and concave). Since $j(0) = 0$, it implies $j(y) \geq 0$ for all $y\geq 0$, in other words $V_1''(z) \geq V_0''(z)$. On the other hand, 
\begin{eqnarray*}
V_0''\po z\pf & \geq & a(x) - \frac12 a''(x)\\
& \geq & \frac12.
\end{eqnarray*}
As a consquence, both $K_\psi(0)$ and $\widetilde K_\psi(0)$ satisfies $\mathcal B(1,4)$ (see for instance \cite[Theorem 5.4.7]{Logsob}, applied to the diffusion with generator $\x^2 - V_\varepsilon \x$).
\end{proof}

To sum up the consequences of the previous results,
\begin{cor}\label{CorAlpha}\
\begin{enumerate}
\item The operators $K_\psi$ and $\widetilde K_\psi$ are $(4,1,1)$-confining and the operator $Q_\psi$ is $(0,\sqrt \delta, 1)$-contractive.
\item The invariant measure $\nu_\psi$ of the embedded chain associated to $\psi(X)$ satisfies $\mathcal B\po 1, \frac{4\sqrt \delta}{1 - \sqrt \delta }\pf$.
\end{enumerate}
\end{cor}
\begin{proof}
The sub-commutation property has been showed in Lemma \ref{LemKalphcont}, and the local inequality is clear for $Q_\psi$ which is deterministic, and is a consequence of Lemma \ref{LemKalphlip} and \ref{LemKalph0} for $K_\psi$ and $\widetilde K_\psi$.

From Lemma \ref{LemConfine}, the transition operator of the embedded chain associated to $\psi(X)$, $P_\psi = K_\psi Q_\psi$, is $(4 \sqrt \delta, \sqrt \delta, 1)$-confining, conclusion follows again from Lemma \ref{LemConfine}.
\end{proof}

\subsubsection{Perturbation and conclusion}\label{Subsub3}

The last step of our procedure is the study of a perturbation of $\nu_\psi$. Since the rate of jump of $Z = \psi(X)$ at point $z$ is $\lambda_\psi(z) = \psi^{-1}(z)$ and the operator $K_\psi$ is such that $K_\psi f\po\psi(x)\pf = \E\po f \po \psi(x+ T_x)\pf\pf$, according to Lemma \ref{lemMueMu}, the invariant measure $\mu_\psi$ of $Z$ is the pertrubation of $\nu_\psi$ by the function $g$ defined by
\begin{eqnarray*}
g\po \psi(x)\pf & = & K_\psi\po \frac{1}{\lambda_\psi}\pf\po \psi(x)\pf\\
& = & \E\po \frac1{x+ T_x}\pf
\end{eqnarray*}
\begin{lem}\label{Lemg}
The function $g$ is decreasing, and $\ln g$ is $\sqrt{\frac{2}{\pi}}$-Lipschitz.
\end{lem}
\begin{proof}
Let 
\[h(x) = \E\po \frac1{x+ T_x}\pf  =  \int_0^\infty e^{-\int_0^t (x+u) du}dt = \int_0^\infty e^{-\frac{t^2}2 - xt}dt,\]
so that $g(z) = h\po \psi^{-1}(z)\pf$. Since $h$ is decreasing and $\psi^{-1}$ is increasing, $g$ is decreasing. Moreover, as $|\po\ln g\pf'(z)| = \sqrt{a \po \psi^{-1}(s)\pf} |(\ln h)'\po \psi^{-1}(z)\pf|$ and $a \leq 1$, it is sufficient to prove $\ln h$ is $\sqrt{\frac{2}{\pi}}$-Lipschitz. Since $h'< 0$ and $h''> 0$, $(\ln h)'$ is negative and increasing: for all $x\geq 0$,
\[0 \hspace{10pt} \geq \hspace{10pt} (\ln h)'(x) \hspace{10pt} \geq \hspace{10pt} \frac{h'(0)}{h(0)} \hspace{10pt} = \hspace{10pt} - \sqrt{\frac{2}{\pi}}.\]
\end{proof}
To apply to $\nu_\psi$ and $g$ the perturbation Lemma \ref{lemPerturbMonotone} of the Appendix, we need to bound $g(m_\psi)$, where $m_\psi$ is the median of $\nu_\psi$, and $\nu_\psi(g^{-1})$. In fact, note that $\nu_\psi$, which is the invariant measure of the embedded chain associated to the process $\psi(X)$, is also the image through the function $\psi$ of $\mu_e$ the invariant measure of the embedded chain associated to the initial process $X$. In particular if $m_e$ is the median of $\mu_e$ then $m_\psi = \psi\po m_e\pf$. Keeping the notation $h(x) = g\po \psi(x)\pf$, we have $g(m_\psi) = h(m_e)$ and $\nu_\psi (g^{-1}) = \mu_e(h^{-1})$.
\begin{lem}\label{LemEstimhnu}
We have
\begin{eqnarray*}
\max\po \frac{h(0)}{h(m_e)}, \mu_e \po h^{-1}\pf \pf & \leq & 3\po 1 + \frac{\delta }{\sqrt{1-\delta^2}}\pf.
\end{eqnarray*}
\end{lem}
\begin{proof}
Recall that, keeping the notations of Section \ref{SectionEmbedded}, if $T_x$ is the first time of jump of the process starting from $x$ and $E$ is a standard exponential variable, then $T_x \overset{dist}= \Lambda_x^{-1}\po E\pf$. In the present case $\Lambda_x(t) = \int_0^t (x+u)du$, so that $T_x \overset{dist}= \sqrt{x^2 + 2E} - x$. In particular if $Y$ is a random variable with measure $\mu_e$, $Y \overset{dist}= \delta \sqrt{Y^2 + 2E}$, 
 so that
\[(1-\delta^2)\E\po Y^2 \pf = 2\delta^2 \E(E) = 2\delta^2.\]
From this,
\[\mathbb P\po Y \geq t \pf \ \leq\ \frac{\delta^2}{(1-\delta^2) t^2},\]
which implies
\[m_e \ \leq\ \frac{\sqrt 2\delta }{\sqrt{1-\delta^2}}. \]
Moreover
\[h(x) \hspace{7pt} = \hspace{7pt} \E\po \frac{1}{x + T_x }\pf \hspace{7pt} \geq \hspace{7pt} \frac{1}{x+2} \mathbb P\po T_x \leq 2 \pf \hspace{7pt} \geq \hspace{7pt}  \frac{1}{x+2} \mathbb P\po T_0 \leq 2 \pf \hspace{7pt} = \hspace{7pt} \frac{1}{x+2}\po 1- e^{-2}\pf.\]
Hence
\[\frac{h(0)}{h(m_e)} \hspace{10pt} \leq \hspace{10pt}  \sqrt{\frac{\pi}{2}}\times\frac{ m_e + 2}{1- e^{-2}} \hspace{10pt}  \leq \hspace{10pt}  3\po 1 + \frac{\delta }{\sqrt{1-\delta^2}}\pf.\]
Finally, if $Y$ is a random variable with law $\mu_e$,
\[ \mu_e \po h^{-1} \pf \hspace{10pt}  = \hspace{10pt} \E \po \frac{1}{h(Y)} \pf \hspace{10pt} \leq \hspace{10pt} \frac1{1-e^{-2}} \po 2 + \sqrt{\E\po Y^2\pf}\pf \hspace{10pt}  \leq \hspace{10pt}  3\po 1 + \frac{\delta }{\sqrt{1-\delta^2}}\pf.\]
\end{proof}

We can now bring the pieces together.

\begin{proof}[proof of Proposition \ref{PropTCPLinea}]
We have proved in Corollary \ref{CorAlpha} that $\nu_\psi$ satisfies a log-Sobolev inequality. From Lemma \ref{Lemg}, \ref{LemEstimhnu} and \ref{lemPerturbMonotone}, the perturbation $\nu_g$ of $\nu_\psi$ defined by $\nu_g f = \frac{1}{\nu_\psi(g)}\nu_\psi\po g f \pf$ also satisfies such an inequality. From Lemma \ref{lemMueMu}, the invariant measure of $\psi(X)$ is $ \nu_g \widetilde K_\psi$, and it also satisfies a log-Sobolev inequality since $\widetilde K_\psi$ is confining (Corollary \ref{CorAlpha}). It means $\mu$, the invariant measure of $X$, satisfies a weighted log-Sobolev inequality
\[\mu \po f^2\ln f^2\pf - \mu (f^2) \ln \mu( f^2)\ \leq \ c \mu\po a |f'|^2\pf.\]
The conditions of Corollary \ref{CorTCPlineLogSob} are fulfilled, and Proposition \ref{PropTCPLinea} is proved.
\end{proof}

\section*{Appendix}
\subsection*{Monotonous perturbation on the half-line}

Let $\nu$ be a probability measure on $\R_+$ with a positive smooth density (still denoted by $\nu$), and $g$ be a positive smooth function on $\R_+$ such that $\nu(g) = 1$. We define $\nu_g$, the perturbation of $\nu$ by $g$, by $\nu_g (f) = \nu(fg)$ for all bounded $f$. Let $m$ be the median of $\nu$, defined by $\nu\po [0,m]\pf = \frac12$.

The aim of this section is to prove the following:


\begin{lem}\label{lemPerturbMonotone}
Suppose $g$ is non-increasing and $g(0) := \underset{x\rightarrow 0}\lim g(x) \neq \infty$.
\begin{enumerate}
\item If $\nu$ satisfies the Poincar\'e inequality $\mathcal B(2,c_1)$, then $\nu_g$ satisfies $\mathcal B(2,c_2)$ with
\[c_2 = 8\frac{g(0)}{g(m)}c_1.\]
\item If $\ln g$ is $\kappa$-Lipschitz and $\nu$ satisfies the log-Sobolev inequality $\mathcal B(1,c_1)$ then $\nu_g$ satisfies $\mathcal B(1,c_2)$ with for all $\varepsilon \in(0,1)$
\[ c_2 \leq  \po\frac{2}{1-\varepsilon} + 8\frac{g(0)}{g(m)}\po 2+  \frac{\frac{c_1\kappa^2}2 + \varepsilon \ln \nu\po g^{1-\frac{1}{\varepsilon}}\pf}{1-\varepsilon}\pf\pf c_1.\]
\end{enumerate}
\end{lem}
\textbf{Remark:} actually as far as point 2 is concerned the monotonicity of $g$ is only needed to get the explicit estimate of $c_2$: as soon as $\nu$ satisfies a log-Sobolev inequality and $\ln g$ is Lipschitz,  $\nu_g$ satisfies a log-Sobolev inequality (see \cite{Aida}).

Moreover when $\nu$ satisfies a log-Sobolev inequality and $\ln g$ is Lipschitz, $\nu \po g^{\alpha}\pf$ is finite for all $\alpha\in\R$ (see \cite{Aida2}), so that $c_2$ is finite.

\begin{proof}[proof of point 1]
According to Muckenhoupt work (see \cite[Theorem 6.2.2 p. 99 and Remark 6.2.3]{Logsob}), a probability with density $h>0$ satisfies $\mathcal B(2,c)$ iff $B_{m_h}(h)$ is finite when $m_h$ is the median of $h(t)dt$ and
\[B_\alpha(h) = \max\po \underset{x\in(\alpha,\infty)}{\sup} \po \int_x^\infty h(t)dt \int_\alpha^x \frac{1}{h(t)}dt\pf, \underset{x\in(0,\alpha)}{\sup} \po \int_0^x h(t)dt \int_x^\alpha \frac{1}{h(t)}dt\pf\pf .\]
Furthermore, in that case, the optimal $c$ (namely the smallest $c$ such that $\mathcal B(2,c)$ holds)  is such that 
\[\frac12\underset{\alpha >0}{\inf}B_\alpha(h) \leq \frac12 B_{m_h}(h)\leq c \leq 4\underset{\alpha >0}{\inf}B_\alpha(h)  \leq 4 B_{m_h}(h).\]
In the present case, for all $x\geq m$,
 \begin{eqnarray*}
   \int_x^\infty g(t) \nu(t)dt \int_{m}^x \frac{1}{g(t)\nu(t)}dt & \leq & \int_x^\infty g(x) \nu(t)dt \int_{m}^x \frac{1}{g(x)\nu(t)}dt\\
   &\leq & 2c_1.
 \end{eqnarray*}
 and  for all $x\leq m$
 \begin{eqnarray*}
  \int_0^x g(t) \nu(t)dt   \int_x^m  \frac{1}{g(t)\nu(t)}dt & \leq & \int_0^x g(0) \nu(t)dt   \int_x^m  \frac{1}{g(m)\nu(t)}dt \\
  & \leq & 2\frac{g(0)}{g(m)}c_1.
 \end{eqnarray*}
 Hence $\nu_g$ satisfies $\mathcal B(2,c_2)$ with
 \[c_2 \leq 4\underset{\alpha >0}{\inf}B_\alpha(\nu g) \leq 4 B_m(\nu g) \leq 8\frac{g(0)}{g(m)}c_1.\]
\end{proof}

\begin{proof}[proof of point 2]
Following a computation of Aida and Shigekawa (\cite{Aida}), we apply the inequality $\mathcal B(1,c_1)$, namely
\begin{eqnarray*}
\forall f\in\A,\hspace{20pt}\nu\po f^2\ln f^2 \pf &  \leq & c_1 \nu (f')^2 + \po \nu f^2\pf \ln \po \nu f^2\pf,
\end{eqnarray*}
to the function $f\sqrt g$, which reads
\begin{eqnarray}\label{EqavantYoung}
\forall f\in\A,\hspace{20pt}\nu_g\po f^2\ln f^2 \pf + \nu_g\po f^2 \ln g \pf & \leq &  c_1 \nu_g \po f' + \frac{g'}{2g}f \pf^2 + \po \nu_g f^2\pf \ln \po \nu_g f^2\pf.
\end{eqnarray}
From the inequality $(a+b)^2 \leq 2a^2 + 2 b^2$ and the assumption on $\ln g$,
\begin{eqnarray*}
\nu_g \po f' + \frac{g'}{2g} f\pf^2 & \leq & 2 \nu_g (f')^2 + \frac{\kappa^2}2  \nu_g (f^2).
\end{eqnarray*}
On the other hand, from the Young inequality $st \leq s\ln s - s + e^t$ applied with $s= \varepsilon f^2$ and $t = -\varepsilon^{-1}\ln \po \frac{ g}{g(0)}\pf $ for any $\varepsilon>0$,
\begin{eqnarray*}
- \nu_g\po f^2 \ln g \pf & = & - \nu_g\po f^2 \ln \po \frac{ g}{g(0)}\pf \pf - \ln g(0) \nu_g \po f^2\pf \\
 & \leq & \varepsilon \nu_g\po f^2 \ln f^2\pf - \po\varepsilon(1-\ln \varepsilon) + \ln g(0)\pf \nu_g\po f^2\pf + \nu_g \po \po\frac{g(0)}{g}\pf^\frac{1}{\varepsilon}\pf.
\end{eqnarray*}
Thus Inequality \eqref{EqavantYoung} yields
\begin{eqnarray*}
(1-\varepsilon)\nu_g\po f^2\ln f^2 \pf  & \leq & 2 c_1 \nu_g (f')^2 + \po \frac{c_1\kappa^2}2  - \varepsilon(1-\ln \varepsilon) - \ln g(0) \pf \nu_g \po f^2\pf + \nu_g \po \po\frac{g(0)}{g}\pf^\frac{1}{\varepsilon}\pf\\
& & + \nu_g(f^2)\ln \nu_g(f^2).
\end{eqnarray*}
Thanks to Gross' Lemma (2.2 of \cite{Gross}), this implies (for $\varepsilon <1$)
\begin{eqnarray}\label{EqapresYoung}
\nu_g\po f^2\ln f^2 \pf  - \nu_g(f^2)\ln \nu_g(f^2) & \leq & \frac{2 c_1}{1-\varepsilon} \nu_g (f')^2 + \gamma \nu_g \po f^2\pf 
\end{eqnarray}
with
\begin{eqnarray*}
\gamma & = &  \frac{\frac{c_1\kappa^2}2  - \varepsilon(1-\ln \varepsilon) - \ln g(0)}{1-\varepsilon} + \frac{\varepsilon}{1-\varepsilon}\po1 + \ln\po \frac{\nu_g\po \po\frac{g(0)}{g}\pf^\frac{1}{\varepsilon}\pf}{\varepsilon}\pf\pf\\
& = & \frac{\frac{c_1\kappa^2}2 + \varepsilon \ln \nu_g\po g^{-\frac{1}{\varepsilon}}\pf}{1-\varepsilon}.
\end{eqnarray*}

It is classical to retrieve a log-Sobolev inequality from Inequality \eqref{EqapresYoung} and a Poincar\'e inequality, thanks to the following inequality (see \cite{Deuschel}, p.146): if $h = f - \nu_g f$,
\[\nu_g\po f^2\ln f^2 \pf  - \nu_g(f^2)\ln \nu_g(f^2)  \leq \nu_g\po h^2\ln h^2 \pf  - \nu_g(h^2)\ln \nu_g(h^2) + 2 \nu_g \po h^2\pf. \]
Together with Inequality \eqref{EqapresYoung} applied to $h$, and since $h' = f'$,
\begin{eqnarray*}
\nu_g\po f^2\ln f^2 \pf  - \nu_g(f^2)\ln \nu_g(f^2) & \leq & \frac{2c_1}{1-\varepsilon} \nu_g (f')^2 +  (\gamma +2) \nu_g \po (f-\nu_g f)^2\pf
\end{eqnarray*}
Since $\nu$ satisfies $\mathcal B(1,c_1)$ it also satisfies $\mathcal B(2,c_1)$. Thus, according to point 1 of Lemma \ref{lemPerturbMonotone}, $\nu_g$ satisfies $\mathcal B\po 2,8\frac{g(0)}{g(m)}c_1\pf$, which means
\begin{eqnarray*}
\nu_g\po f^2\ln f^2 \pf  - \nu_g(f^2)\ln \nu_g(f^2) & \leq & \po\frac{2}{1-\varepsilon} + 8\frac{g(0)}{g(m)}(2+\gamma)\pf c_1 \nu_g (f')^2.
\end{eqnarray*}
\end{proof}

\bibliographystyle{plain}
\bibliography{biblio}
\end{document}